\documentclass{amsart}

\usepackage{amssymb}
\usepackage{latexsym}
\usepackage{amsmath}
\usepackage{euscript}
\usepackage{graphics}

\newtheorem{theorem}{Theorem}[section]
\newtheorem{corollary}[theorem]{Corollary}
\newtheorem{lemma}[theorem]{Lemma}
\newtheorem{proposition}[theorem]{Proposition}
\theoremstyle{definition}
\newtheorem{definition}[theorem]{Definition}
\theoremstyle{remark}
\newtheorem{remark}[theorem]{Remark}
\numberwithin{equation}{section}
\renewenvironment {proof} {\begin{trivlist} \item[\hspace{\labelsep}%
\sc Proof.]}{$\Box$ \end{trivlist}}

      \def\dC{{\mathbb C}}

   \def\dN{{\mathbb N}}   
      \def\dR{{\mathbb R}}

   \def\dZ{{\mathbb Z}}

   \def\cH{{\mathcal H}}   
      
\def\cM{{\mathcal M}}   \def\cN{{\mathcal N}}   
      
\def\cS{{\mathcal S}}      \def\cU{{\mathcal U}}

\newcommand {\wt}{\widetilde}
\newcommand {\wh}{\widehat}

\begin{document}

\title[Schur algorithm for Stieltjes indefinite moment problem ]
{Schur algorithm for Stieltjes indefinite moment problem }

\author{Vladimir Derkach}
\address{Department of Mathematics, Dragomanov National Pedagogical University, Kiev, Pirogova 9, 01601, Ukraine}
 \address{Department of Mathematics, Donetsk National University, 600-Richchya Str 21, Vinnytsya, 21021,  Ukraine}
 \email{derkach.v@gmail.com}
\author{Ivan Kovalyov}
 \address{Department of Mathematics, Dragomanov National Pedagogical University,
Kiev, Pirogova 9, 01601, Ukraine}
\email{i.m.kovalyov@gmail.com }

\thanks{This work was supported  by grants of Ministry of Education and
Science of Ukraine (project numbers:  0115U000136, 0115U000556); V.D. is indebted to the German Science Foundation (DFG) for support under Grant  TR 903/16-1 }

\begin{abstract}

Nondegenerate truncated indefinite Stieltjes moment problem in the class
$\mathbf{N}_{\kappa}^{k}$ of generalized Stieltjes functions is considered.
To describe the set of solutions of this problem we apply  the Schur step-by-step algorythm, which leads to the expansion of these solutions in generalized Stieltjes continuous fractions studied recently in~\cite{DK15}. Explicit formula for the resolvent matrix in terms of generalized Stieltjes polynomials  is found.

\end{abstract}

\maketitle

\section{Introduction}
Classical Stieltjes moment problem consists in the following: given a sequence of real numbers $s_{j}$  $(j\in\mathbb{Z}_{+}:=\mathbb{N}\cup\{0\})$
find a  positive measures $\sigma$ with a support on $\mathbb{R}_{+}$, such that
\begin{equation}\label{int1}
    \int_{\mathbb{R}_{+}}t^jd\sigma(t)=s_{j}\qquad (j\in\mathbb{Z}_{+}).
\end{equation}
In \cite{St89} T. Stieltjes described piecewise solutions $\sigma$
of this problem in connection with small vibration problem for a
massless thread with a countable set of point masses. Full
description of all positive measures $\sigma$, which
satisfy~\eqref{int1}, was given by M.G.~Kre\u{\i}n in~\cite{Kr52}.
The problem~\eqref{int1}, when $\sigma$ is recovered from a finite
sequence $\{s_{j}\}_{j=0}^{2n}$ is called the truncated Stieltjes
moment problem and was studied in~\cite{Kr67}.

By the Hamburger--Nevanlinna theorem~\cite{Akh} the truncated Stieltjes moment problem can be reformulated in terms
of the Stieltjes transform
\begin{equation}\label{int3}
    f(z)=\int_{\mathbb{R}_{+}}\frac{d\sigma(t)}{t-z}\qquad z\in \mathbb{C}\backslash \mathbb{R}_{+}
\end{equation}
of $\sigma$ as the following interpolation problem at $\infty$
\begin{equation}\label{int4}
    f(z)=-\frac{s_{0}}{z}-\frac{s_{1}}{z^2}-\cdots-\frac{s_{2n}}{z^{2n+1}}
    +o\left(\frac{1}{z^{2n+1}}\right),\quad\quad z\widehat{\rightarrow}\infty.
\end{equation}
The notation $z\widehat{\rightarrow}\infty$ means that $z\rightarrow\infty$  nontangentially, that is inside the
sector $\varepsilon<\arg z<\pi-\varepsilon$ for some
$\varepsilon>0$. It follows easily from~\eqref{int1} that the inequalities
\begin{equation}\label{int5}
    S_{n+1}:=\left(s_{i+j}\right)_{i,j=0}^{n}\geq0, \quad S_{n}^{+}:=\left(s_{i+j+1}\right)_{i,j=0}^{n-1}\geq0
\end{equation}
are necessary for solvability of the moment problem~\eqref{int4}. If the matrices $S_{n+1}$ and $S_{n}^{+}$ are nondegenerate, then  the inequalities $S_{n+1}>0$, $S^{+}_{n}>0$ are also sufficient for solvability of the moment problem~\eqref{int4}, see~\cite{Kr67}. The degenerate case is more subtle and was studied in~\cite{CF91}.

The function $f$ in \eqref{int3} belongs to the class $\mathbf{N}$ of functions holomorphic on $\mathbb{C}\backslash \mathbb{R}$ with nonnegative imaginary part in  $\mathbb{C}_{+}$ and such that $f(\overline{z})=\overline{f(z)}$ for $z\in \mathbb{C}_{+}$.

Moreover, $f$ belongs to the Stieltjes class $\mathbf{S}$ of functions $f\in \mathbf{N}$, which admit holomorphic and nonnegative continuation to $\mathbb{R}_{-}$. By M.G. Krein criterion, see~\cite{KN77}
\begin{equation}\label{int7}
    f\in \mathbf{S}\Longleftrightarrow f\in \mathbf{N} \quad\mbox{and}\quad zf\in \mathbf{N}.
\end{equation}

Indefinite version of the class $\mathbf{N}$ was introduced in~\cite{KL73}.
\begin{definition}\label{def:Nk}{\rm~\cite{KL73}}
    A function $f$ meromorphic on $\mathbb{C}\backslash\mathbb{R}$ with the set of holomorphy ${\mathfrak h}_f$
    is said to be in the generalized Nevanlinna class $\mathbf{N}_{\kappa}$ $(\kappa\in\mathbb{N})$, if for every set $z_{j}\in \mathbb{C}_{+}\cap{\mathfrak h}_f$ ($j=1,\ldots,n$) the form
\[
    \sum_{i,j=1}^{n}\frac{f(z_{i})-\overline{f(z_{j})}}{z_{i}-\overline{z}_{j}}
    \xi_{i}\overline{\xi}_{j}
\]
has at most $\kappa$ and for some choice of $z_{j}$ ($j=1,\ldots,n$) exactly $\kappa$ negative squares.
\end{definition}
 The generalized  Stieltjes class $\mathbf{N}_{\kappa}^{+}$ was defined in \cite{KL77} as the class of functions $f\in \mathbf{N}_{\kappa}$, such that $zf\in \mathbf{N}$. Similarly, in \cite{D91,D95} the class $\mathbf{N}_{\kappa}^{k}$  ($\kappa, k\in \mathbb{N}$) was introduced as the set of functions $f\in \mathbf{N}_{\kappa}$, such that $zf\in \mathbf{N}_{k}$.

In \cite{KL79} the moment problem in the class $\mathbf{N}_{\kappa}$ ($\mathbf{N}_{\kappa}^{+}$)  was considered in the following setting: Given a real sequence $\{s_{j}\}_{j=0}^{\infty}$, find $f\in \mathbf{N}_{\kappa}$ ($\mathbf{N}_{\kappa}^{+}$) such that \eqref{int4} holds for every $n\in \mathbb{N}$. In particular, it was shown in \cite{KL79} that the problem \eqref{int4} is solvable in $\mathbf{N}_{\kappa}^{+}$ if the number $\nu_{-}(S_{n})$ of negative eigenvalues of $S_{n}$ does not exceed $\kappa$ for all $n$ big enough and $S_{n}^{+}>0$ for all $n\in \mathbb{N}$.
The indefinite moment problem in generalized Stieltjes class $\mathbf{N}_{\kappa}^{k}$ was studied in \cite{D97}.

In the present paper we consider the following truncated indefinite moment problem.

\noindent{\bf Problem} $MP_{\kappa}^{k}(\textbf{s},  \ell)$.
Given $\ell,\kappa,k\in\mathbb{Z}_{+}$,  and a sequence $\textbf{s}=\left\{s_{j}\right\}_{j=0}^{ \ell}$ of real numbers, describe the set $\cM_{\kappa}^{k}({\mathbf s})$  of functions $f\in \mathbf{N}_{\kappa}^{k}$, which satisfy the asymptotic expansion
\begin{equation}\label{3p.2.3}
    f(z)=-\frac{s_{0}}{z}-\cdots-\frac{s_{\ell}}{z^{\ell+1}}+
o\left(\frac{1}{z^{\ell+1}}\right)\quad
(z=iy, \,\, y{\uparrow}\infty).
\end{equation}
 Such a moment problem is called {\it even} or {\it odd} regarding to the oddness of the number $\ell+1$ of moments.
To study this problem  we use the Schur algorithm, which was elaborated in \cite{Der03}, \cite{DD04} and \cite{ADL04}
for the class $\mathbf{N}_{\kappa}$. Let us explain it for the even case, i.e. when ${\bf s}=\{s_{i}\}_{i=0}^{2n-1}$.
Recall, that a number $n_{j}\in \mathbb{N}$ is called a {\it normal index} of the sequence ${\bf s}$, if det$\,S_{n_{j}}\neq0$. The ordered set of normal indices
\[
n_1<n_2<\dots<n_N
\]
 of the sequence ${\bf s}$ is denoted by $\cN({\bf s})$. For every $n_{j}\in \cN(s)$ polynomials of the first and the second kind $P_{n_{j}}(z)$ and $Q_{n_{j}}(z)$ can be defined by standard formulas, see~\eqref{P3.mom.1}.
A sequence ${\bf s}$ is called {\it regular} (see~\cite{DK15}), if
\begin{equation}\label{eq:int_p}
    P_{n_{j}}(0)\neq0\qquad \mbox{for}\quad (1\le j\le N).
\end{equation}
The latter condition is equivalent to the condition $\det S_{n_j}^+\ne 0$ for all $j=1,\dots, N$.

If the set $\cN({\bf s})$ consists of $N$ indices $\cN({\bf s})=\{n_j\}_{j=1}^N$ and $n=n_N$, a function $f\in\cM_{\kappa}^{k}({\bf s})$ can be expanded into a $P$-fraction
\begin{equation}\label{eq:Pfrac}
    - \frac{b_{0}}{\displaystyle {a}_{0}(z)-\frac{b_{1}}{\displaystyle
    {a}_{1}(z)-\dots-\frac{b_{N-1}}{\displaystyle {a}_{N-1}(z)+\tau(z)}}},
\end{equation}
where $b_{j}(\ne 0)$ are real numbers and ${a}_j$ are monic
polynomials  of degree $k_j=n_{j+1}-n_j$, by using $N$ steps of the Schur algorythm,  see~\cite{DD04}. $P$-fractions were introduced and studied in~\cite{Mag62}, see also \cite{Peh92}. In the present paper we show that for  $f\in\cM_{\kappa}^{k}({\bf s})$ with regular ${\bf s}$ one step of the Schur algorythm can be splitted into two substeps, which lead to the following representation of $f$
\begin{equation}\label{eq:Schur_substeps}
    f(z)=\frac{1}{-zm_{1}(z)+\frac{\displaystyle1}{\displaystyle l_{1}+f_1(z)}},
\end{equation}
where $m_{1}(z)$ is a polynomial, $l_{1}\in\mathbb{R}\backslash\{0\}$, $f_1\in{\mathbf N}_{\kappa-\kappa_1}^{k-k_1}$, and $\kappa_1=\nu_-(S_{n_1})$, $k_1=\nu_-(S_{n_1}^+)$.
By iterating this process, we show that for  ${\bf s}=\{s_{i}\}_{i=0}^{2n-1}$ and
$\cN({\bf s})=\{n_j\}_{j=1}^N$ the problem $MP_{\kappa}^{k}({\mathbf s})$ is solvable, if and only if
\begin{equation}\label{eq:MP_kk}
    \kappa_N:=\nu_-(S_N)\le\kappa, \quad k_N:=\nu_-(S_N^+)\le k,
\end{equation}
and every solution $f\in \cM_{\kappa}^{k}({\mathbf s})$ admits the representation
as the continued fraction
\begin{equation}\label{eq:con_frac}
    f(z)=\frac{1}{-zm_{1}(z)+\frac{\displaystyle1}{\displaystyle l_{1}+\cdots+\frac{\displaystyle1}{\displaystyle-zm_{N}(z)
    +\frac{\displaystyle1}{\displaystyle l_N+\tau(z)}}}},
\end{equation}
where $m_{j}$ are  polynomials, $l_{j}\in\mathbb{R}\backslash\{0\}$ and $\tau$ is a parameter function from some generalized Stieltjes class ${\mathbf N}_{\kappa-\kappa_N}^{k-k_N}$, such that $\tau(z)=o(1)$ az $z\wh\to\infty$. Such continued fractions were studied in~\cite{DK15}.

Associated with the continued fraction~\eqref{eq:con_frac} is a system of difference equations
\begin{equation}\label{eq:Dif_sys}
    \left\{
    \begin{array}{l}
    y_{2j}-y_{2j-2}=l_{j}y_{2j-1},\\
    y_{2j+1}-y_{2j-1}=-zm_{j+1}(z)y_{2j}\\
    \end{array}\right.
\end{equation}
see~\cite[Section~1]{Wall}.
Following \cite{St89} (see also~\cite[Section~5.3]{KN77}, \cite{DM95}) we introduce
Stieltjes polynomials $P_j^+$ and $Q_j^+$ in such a way, that
$u_j=Q_j^+$ and $v_j=P_j^+$  are solutions of the system \eqref{eq:Dif_sys} subject to the initial conditions
\begin{equation}\label{s4.4'!}
    u_{-1}\equiv -1, \quad u_{0}\equiv0; \qquad v_{-1}\equiv0, \quad u_{0}\equiv1.
\end{equation}
This implies that
the convergents $\frac{u_j}{v_j}$ of the  continued fraction~\eqref{eq:con_frac}
take the form
\begin{equation}\label{eq:conv}
    \frac{u_j}{v_j}=\frac{Q_j^+}{P_j^+}\quad(j=1,\dots , 2N).
\end{equation}
In view of~\eqref{eq:conv} the representation~\eqref{eq:con_frac} can be rewritten as
\begin{equation}\label{eq:des_sol}
    f(z)=\frac{Q^{+}_{2N-1}(z)\tau(z)+Q^{+}_{2N}(z)}{P^{+}_{2N-1}(z)\tau(z)+P^{+}_{2N}(z)},
\end{equation}
Moreover, the solution matrix, i.e. the $2\times 2$ matrix $W_{2N}(z)$ of coefficients of the linear-fractional transform~\eqref{eq:des_sol} admits the factorization
\begin{equation}\label{eq:sol_matrix}
    W_{2N}(z)=\begin{pmatrix}
        Q^{+}_{2N-1}(z) & Q^{+}_{2N}(z) \\
        P^{+}_{2N-1}(z) & P^{+}_{2N}(z) \\
      \end{pmatrix}
      =M_{1}(z)L_{1}\ldots M_{N}(z)L_N,
\end{equation}
where the matrices $M_j$ and $L_j$ are defined by
\begin{equation}\label{eq:LMj}
M_{j}(z)=\begin{pmatrix}
1 & 0 \\
-zm_{j}(z) & 1 \\
\end{pmatrix},\quad\mbox{and}\quad
L_{j}=\begin{pmatrix}
1 & l_{j} \\
0 & 1 \\
\end{pmatrix}\quad j=\overline{1,N}.
\end{equation}

In the case when the sequence ${\mathbf s}$ satisfies the conditions
\begin{equation}\label{eq:MP_k+}
    S_N>0, \quad S_N^+>0,
\end{equation}
${\mathbf s}$  is automatically regular in the sense of~\eqref{eq:int_p} and $m_j$, $l_j$ are positive numbers, In this case the system~\eqref{eq:Dif_sys} describes  small vibrations of a massless thread with  masses $m_j$ and distances $l_j$ between them, see~\cite[Appendix]{Akh}. The case, when $ S_N^+>0$ and $\nu_-(S_N)>0$ was studied by M.G~Kre\u{\i}n and H.~Langer~\cite{KL79}. In this case it may happen that $m_j$ is either a negative real or even a polynomial of degree  1, and the system~\eqref{eq:Dif_sys} was interpreted in~\cite{KL79} as a generalized Stieltjes string with negative masses and dipoles. In the general case, when ${\bf s}$ is a regular sequence and all $l_j$ are positive, one can treat system~\eqref{eq:Dif_sys} as a generalized Stieltjes string with multipoles, cf.~\cite{D97}.

Continued fractions of the form~\eqref{eq:con_frac} with negative masses $m_j$ were studied by Beals, Sattinger and
Szmigielski~\cite{Beals_Sat_Szm} in connection with the theory of multi-peakon solutions of the Camassa-Holm equation. In particular, they noticed that in the indefinite case, the inverse problem is not always solvable in the class of such continued fractions. In~\cite{Eck_Kost_14}
 it was shown that the inverse spectral problem for multi-peakon solutions of the Camassa-Holm equation is solvable in the class of continued fractions of the form~\eqref{eq:con_frac} with polynomials $m_j(z)=d_jz+m_j$ of formal degree 1  ($d_j\ge 0$, $m_j\in\dR$). These result is in the full correspondence with the description of solutions of the Stieltjes indefinite moment problem given in~\cite{KL79}.

A description of the set of solutions of odd Stieltjes moment problem, corresponding to a sequence $\textbf{s}=\left\{s_{j}\right\}_{j=0}^{2n-2}$, is also found in a form similar to~\eqref{eq:des_sol}. If $\cN({\mathbf s})=\{n_j\}_{j=1}^N$ and $n=n_N$, then the factorization formula for the corresponding solution matrix $W_{2N-1}$ takes the form
\[
 W_{2N-1}(z)=M_{1}(z)L_{1}\ldots L_{N-1}M_{N}(z).
\]
In the case of a non-regular sequence ${\mathbf s}$ every solution $f\in \cM_{\kappa}^{k}({\mathbf s})$ admits an expansion in a
continued fraction of type~\eqref{eq:con_frac}, where $l_j$ are polynomials.
The corresponding results will be published elsewhere. Notations in the present paper are quite tricky: all the objects which appear on the $j$-th step are endowed with the index $j$, regardless to the substep. To make difference between substeps, the moments which appear on the 1-st substep are denoted by Fraktur script, while moments which appear on the 2-nd substep are denoted by Latin script. The only exception is made for the solution matrix - the solution matrix, corresponding to an odd Stieltjes moment problem is denoted by $W_{2N-1}$, while solution matrix, corresponding to an even Stieltjes moment problem is denoted by $W_{2N}$.

Now, briefly describe the content of the paper. Section 2 contains some preliminary statements concerning the class ${\mathbf N}_\kappa^k$ of generalized Stieltjes functions,  
class ${\cU}_\kappa(J)$ of generalized $J$-unitary matrix functions, normal
indices of finite real sequences 
and some inversion formulas for asymptotic expansions.
Solutions to odd and even basic moment  problems
will be described in Section 3. Section 4 presents a general
Schur recursion algorithm, which allows to parametrize solutions of odd and even Stieltjes indefinite moment problems ${MP}_\kappa({\bf s},2n_N-2)$ and ${MP}_\kappa({\bf s},2n_N-1)$, respectively. Factorization formulas for  solution matrices $W_{2N-1}$ and $W_{2N}$ for odd and even Stieltjes indefinite moment problems based on the Schur algorithm are found. In Section 5 we introduce Stieltjes polynomials and find explicit formulas for solution matrices $W_{2N-1}$ and $W_{2N}$  in terms of Stieltjes polynomials.

\section{Preliminaries}
\subsection{Generalized Nevanlinna and Stieltjes classes}
The class $\mathbf{N}_{\kappa}$, introduced in
Definition~\ref{def:Nk} is called the generalized Nevanlinna  class.
For $f\in \mathbf{N}_{\kappa}$ let us write $\kappa_{-}(f)=\kappa$.
In particular, if $\kappa=0$ then the class $ \mathbf{N}_{0}$
coincides with the class $\mathbf{N}$ of Nevanlinna functions
(see~\cite{KN77}).

Every real polynomial
$P(t)=p_{\nu}t^{\nu}+p_{\nu-1}t^{\nu-1}+\ldots+p_{1}t+p_{0}$ of
degree $\nu$ belongs to a class $\mathbf{N}_{\kappa}$, where the
index $\kappa=\kappa_{-}(P)$ can be evaluated by
(see~\cite[Lemma~3.5]{KL77})
\begin{equation}\label{3p.kappaP}
\kappa_{-}(P)=\left\{
\begin{array}{cl}
\left[\frac{\nu+1}{2}\right],&        \mbox{ if } p_{\nu}<0; \mbox{ and } \nu  \mbox{ is odd };\\
\left[\frac{\nu}{2}\right],&  \mbox{ otherwise } .
\end{array}
\right.
\end{equation}

Denote by $\nu_{-}(S)$ $(\nu_{+}(S))$ the number of negative (positive, resp.) eigenvalues of the matrix $S$.
Let  ${\cH}$ be the set of finite real sequences
${\bf s}=\{s_j\}_{j=0}^{\ell}$ and let
 ${\cH}_{\kappa, \ell}$ be the set of sequences
${\bf s}=\{s_j\}_{j=0}^{\ell}\in{\cH}$, such that
\begin{equation}\label{eq:Hkl}
   \nu_{-}(S_n)=\kappa \quad(n=[\ell/2]+1)
\end{equation}
where $S_n$  is defined by~\eqref{int5}.
The index $\nu_{-}(S_n)$ for a Hankel matrix $S_n$ can be calculated
by the Frobenius rule (see~\cite[Theorem~X.24]{Gant64}). In
particular, if all the determinants $D_j:=\det S_j$ ($j\in\dZ_+$) do
not vanish, then $\nu_{-}(S_n)$ coincides with the number of sign
alterations in the sequence
\[
D_0:=1,\quad D_1,\quad D_2,\dots,\quad D_n.
\]

Let us remind some statements concerning  the classes $\mathbf{N}_{\kappa}$ and ${\cH}_{\kappa, \ell}$  from \cite{KL77,KL79}.
\begin{proposition}(\cite{KL77}) \label{prop:2.1}
Let  $f\in \mathbf{N}_{\kappa}$, $f_{1}\in \mathbf{N}_{\kappa_{1}}$, $f_{2}\in \mathbf{N}_{\kappa_{2}}$. Then
\begin{enumerate}
  \item [(1)] $-{f}^{-1}\in \mathbf{N}_{\kappa}$;
\item [(2)] $f_{1}+f_{2}\in \mathbf{N}_{\kappa'}$, where
  $\kappa'\leq\kappa_{1}+\kappa_{2}$;
\item [(3)] If, in addition,  $f_{1}(iy)=o(y)$ as
$y\rightarrow\infty$ and $f_{2}$ is a polynomial,  then \begin{equation}\label{f_1+f_2}
f_{1}+f_{2}\in \mathbf{N}_{\kappa_{1}+\kappa_{2}}.
\end{equation}
\item[(4)] If a function $f\in \mathbf{N}_{\kappa}$ has an asymptotic
expansion~\eqref{3p.2.3},
then there exists $\kappa'\le\kappa$, such that
$\{s_j\}_{j=0}^{\ell}\in{\cH}_{\kappa', \ell}$.
  \end{enumerate}
\end{proposition}

Recall,  that a Nevanlinna function $f$  is said to be from the Stieltjes
class $\mathbf{S}^+$ ($\mathbf{S}^-$), if it admits a holomorphic and nonnegative (nonpositive, resp.) continuation to the negative half-line. By the M.G. Kre\u{\i}n criterion (\cite{Kr52})
\[
f\in\mathbf{S}^\pm\Leftrightarrow
f\in\mathbf{N}\mbox{ and }z^{\pm 1}f(z)\in\mathbf{N}.
\]

The following generalization of the class $\mathbf S^+$ was introduced in~\cite{D91, D97}.
A function $f\in \mathbf{N}_{\kappa}$  is said to be from the generalized Stieltjes
class $\mathbf{N}_{\kappa}^{\pm k}$, 
if $z^{\pm 1}f(z)\in \mathbf{N}_k$  $\left(\kappa,k\in\mathbb{Z}_{+}\right)$.
In the case $\kappa=k=0$ the class $\mathbf{N}_{0}^{\pm 0}$ coincides with the class $\mathbf{S}^{\pm}$.
The classes $\mathbf{N}_{\kappa}^{\pm}:=\mathbf{N}_{\kappa}^{\pm 0}$ and
$\mathbf{S}^{\pm k}:=\mathbf{N}_{0}^{\pm k}$ were studied in~\cite{KL79} and \cite{DM87,DM97}, respectively.

Denote by ${\cH}_{\kappa, \ell}^k$ the set of real sequences
${\bf s}=\{s_j\}_{j=0}^{\ell}\in{\cH}_{\kappa, \ell}$, such that
$\{s_{j+1}\}_{j=0}^{\ell-1}\in{\cH}_{k, \ell-1}$, i.e.
\begin{equation}\label{eq:Hklk}
   \nu_{-}(S_{[(\ell+1)/2]})=k.
\end{equation}
\begin{proposition}(\cite{KL77})\label{prop:Nkk} The following equivalences hold:
\begin{enumerate}
  \item [(1)] $f\in \mathbf{N}_{\kappa}^k\Longleftrightarrow-\frac{1}{f}\in \mathbf{N}_{\kappa}^{-k}$;
\item [(2)] $f\in \mathbf{N}_{\kappa}^k\Longleftrightarrow zf(z)\in \mathbf{N}_{k}^{-\kappa}$, in particular,
$f\in \mathbf{N}_{\kappa}^+\Longleftrightarrow zf(z)\in \mathbf{S}^{-\kappa}$;
\item [(3)]
 If a function $f\in \mathbf{N}_{\kappa}^k$ has an asymptotic expansion~\eqref{3p.2.3}
then 
\begin{equation}\label{eq:solv}
    \{s_j\}_{j=0}^{\ell}\in{\cH}_{\kappa', \ell}^{k'}\quad\mbox{ with }\kappa'\le\kappa,\quad k'\le k.
\end{equation}
  \end{enumerate}
\end{proposition}

\subsection{Normal indices}

Let $\mathcal{N}(\textbf{s})=\{n_{j}\}_{j=1}^{N}$ be the set of normal indices
of the sequence $\textbf{s}=\{s_j\}_{j=0}^{\ell}$, defined by the
properties
\begin{equation}\label{3p.2.4}
    \det S_{n_{j}}\neq0\quad (j\in\{1,2,\dots,N\}).
\end{equation}
and enumerated in the increasing order.
It follows from the Sylvester identity (see~\cite[Proposition~3.1]{DK15}), that $\mathcal{N}(\textbf{s})$ is the
union of two not necessarily disjoint subsets
\begin{equation}\label{eq:3pN=mu+nu}
    \mathcal{N}(\textbf{s})= \{\nu_{j}\}_{j=1}^{N_{1}}\cup \{\mu_{j}\}_{j=1}^{N_{2}},
\end{equation}
which are selected by
 \begin{equation}\label{eq:3p.nu}
     \det S_{\nu_{j}}\neq0\quad\mbox{and}\quad
    \det S_{\nu_{j}-1}^{+}\neq0,\quad\mbox{ for all
    }j=\overline{1,N_{1}}
\end{equation}
and
\begin{equation}\label{eq:3p.mu}
     \det S_{\mu_{j}}\neq0\quad\mbox{and}\quad
    \det S_{\mu_{j}}^{+}\neq0,\quad\mbox{ for all
    }j=\overline{1,N_{2}}.
\end{equation}
Moreover, the normal indices
$\nu_{j}$ and  $\mu_{j}$ satisfy the following inequalities
\begin{equation}\label{eq:3p.nu<mu}
     0<\nu_{1}\leq\mu_{1}<\nu_{2}\leq\mu_{2}<\ldots
\end{equation}
\begin{corollary}\label{cor:munu}
     If a function $f\in \mathbf{N}_{\kappa}^k$ has an asymptotic expansion~\eqref{3p.2.3}
     with $\ell=2\mu_j-1$ and $\mu_j$ satisfy~\eqref{eq:3p.mu}, then
\begin{equation}\label{eq:ineq_mu}
    \nu_-(S_{\mu_{j}})\leq\kappa,\quad \nu_-(S_{\mu_{j}}^+)\leq k.
\end{equation}
 If a function $f\in \mathbf{N}_{\kappa}^k$ has an asymptotic expansion~\eqref{3p.2.3}
 with $\ell=2\nu_j-2$ and $\nu_j$ satisfy~\eqref{eq:3p.nu}, then
\begin{equation}\label{eq:ineq_nu}
    \nu_-(S_{\nu_{j}})\leq\kappa,\quad \nu_-(S_{\nu_{j}-1}^+)\leq k.
\end{equation}
\end{corollary}
Notice, that the number $\nu_{1}$ can be found from the conditions
\begin{equation}\label{3p.3.1}
    s_{0}=\ldots=s_{\nu_{1}-2}=0,\quad s_{\nu_{1}-1}\neq0,
\end{equation}
since  for such $\nu_{1}$ one has $\det S_i=0$ for $i\le \nu_1$ and
\begin{equation}\label{2p.3.2}
    \det S_{\nu_{1}}
    \neq0\quad\mbox{and}\quad
    \det S_{\nu_{1}-1}^{+}
    \neq0 .
\end{equation}
Therefore, the first normal index of $\textbf{s}$ coincides with $\nu_{1}$, i.e. $n_1=\nu_{1}$.
\subsection{Toeplitz matrices and asymptotic expansions}
A sequence $(c_0,\ldots, c_n)$ of real
numbers determines an upper triangular Toeplitz matrix
$T(c_0,\ldots, c_n)$ of order $(n+1)\times(n+1)$ with entries
$t_{i,j}=c_{j-i}$ for $i\le j$ and $t_{i,j}=0$ for $i>j$:
\begin{equation}\label{TSid1}
T(c_0,\ldots, c_n)=\left(
 \begin{array}{cccccc}
         c_0 &  \dots      & c_n \\
          & \ddots & \vdots \\
       &  & c_{0} \
 \end{array} \right).
\end{equation}
Some of the calculations of the present paper will be based on the following
\begin{lemma}\label{lem:2.3}
Let the functions $c$ and $d$ (meromorphic on $\dC\setminus\dR$)
have the asymptotic expansions
\begin{equation}\label{AsExp}
\begin{split}
   c({z})&=
c_0+\frac{c_{1}}{{z}}+\dots+\frac{c_{n}}{{z}^{n}}
+o\left(\frac{1}{{z}^{n}}\right), \quad{z}\widehat{\rightarrow}\infty;\\
d({z})&=
d_0+\frac{d_{1}}{{z}}+\dots+\frac{d_{n}}{{z}^{n}}
+o\left(\frac{1}{{z}^{n}}\right),
\quad{z}\widehat{\rightarrow}\infty.
\end{split}
\end{equation}
and let $c({z})d({z})=1$.
 Then the Toeplitz matrices $T(c_0,\ldots, c_n)$ and $T(d_0,\ldots, d_n)$ are connected by
 \begin{equation}\label{cdExp}
 T(c_0,\ldots, c_n) T(d_0,\ldots, d_n)=I_{n+1}.
\end{equation}
\end{lemma}

Assume that a sequence $\textbf{s}=\{s_{j}\}_{j=0}^{\ell}$ satisfies
the conditions~\eqref{3p.3.1} with $\nu_1$ replaced by $\nu$, i.e.
\begin{equation}\label{3p.3.1A}
    s_{0}=\ldots=s_{\nu-2}=0,\quad s_{\nu-1}\neq0.
\end{equation}
 If $\ell\ge 2\nu-1$ then one can  define a polynomial $a$ and a constant $b$ by
\begin{equation}\label{eq:2.5a}
a(z)=\frac{1}{D_{\nu}}
    \begin{vmatrix}
        s_{0} & \ldots & s_{\nu-1} & s_{\nu_{}} \\
        \cdots & \cdots & \cdots & \cdots \\
        s_{\nu_{}-1} & \ldots & s_{2\nu_{}-2} & s_{2\nu_{}-1} \\
        1 & z & \ldots & z^{\nu_{}} \\
    \end{vmatrix},\quad
    b=s_{\nu-1}.
\end{equation}
In the case when $\ell= 2\nu-2$ let us set $s_{2\nu-1}$ to be an
arbitrary real number. This number impacts only the last coefficient
$a_0$ of the polynomial
\begin{equation}\label{eq:a}
    a(z)=a_\nu z^\nu+    \dots+a_1z+a_0.
\end{equation}
The following lemma  is a direct corollary of Lemma~\ref{lem:2.3}. It collects some statements concerning asymptotic expansions of the reciprocal function from~\cite[Lemma~2.1]{Der03} and~\cite[Lemma~2.13, Lemma~A3]{DHS12}.
\begin{lemma}\label{lem:D01}
Assume that a sequence $\textbf{s}=\{s_{j}\}_{j=0}^{\ell}$ satisfies
the conditions~\eqref{3p.3.1A} with $\ell\ge 2\nu-1$, let
$\cN(\textbf{s})=\{n_j\}_{j=1}^N$, $n=[\ell/2]$ and let $b$ and the
polynomial $a({z})=\sum_{j=0}^{\nu} a_j{z}^j$ be defined by
\eqref{eq:2.5a}. Then a function $f$ (meromorphic on
$\dC\setminus\dR$) admits the asymptotic expansion
\begin{equation}\label{eq:2.4}
     f({z})=
 -\frac{s_{\nu-1}}{{z}^{\nu}}-\dots-\frac{s_{\ell}}{{z}^{\ell+1}}
  +o\left(\frac{1}{{z}^{\ell+1}}\right), \quad {z} \wh \to \infty,
\end{equation}
if and only if the function $-1/f({z})$ admits the
asymptotic expansion
\begin{equation} \label{eq:2.4a}
-\frac{1}{f({z})} =  \frac{a({z})}{b}+\wt g({z}),\quad {z} \wh \to
\infty,
\end{equation}
where $g({z})$ satisfies one of
the following conditions:
\begin{enumerate}\def\labelenumi{(\roman{enumi})}
\item if $\ell=2\nu-2$ and $s_{2\nu-1}$ in \eqref{eq:2.5a} is an arbitrary real
number, then $\wt g({z})=o({z})$, ${z} \wh \to \infty$;
\item if $\ell=2\nu-1$ then $\wt g({z})=o(1)$ as ${z} \wh \to \infty$;

\item if $\ell>2\nu-1$ then $g({z})$ has the asymptotic expansion
\begin{equation} \label{eq:2.4E}
 \wt g({z})=
 -\frac{{\mathfrak s}_{0}}{{z}}-\dots-\frac{{\mathfrak s}_{\ell-2\nu}}{{z}^{\ell-2\nu+1}}
  +o\left(\frac{1}{{z}^{\ell-2\nu+1}}\right), \quad {z} \wh \to \infty,
\end{equation}
where the sequence $({\mathfrak s}_i)_{i=0}^{\ell-2\nu}$ is determined by
the matrix equation
\begin{equation}\label{eq:2.4G}
     T(\frac{a_{\nu}}{b},\ldots,\frac{a_0}{b},-{\mathfrak s}_0,\ldots,-
 {\mathfrak s}_{\ell-2\nu}) \, T(s_{\nu-1},\dots,s_\ell) = I_{\ell-\nu+2}.
\end{equation}
Moreover, the matrices $\cS_{p}=({\mathfrak s}_{i+j})_{i,j=0}^{p-1}$ are connected with matrices $S_{p+\nu}$ by the equalities
   \begin{equation} \label{2.9}
{\cS}_{p}=(TS_{p+\nu}T^*)^{-1}\quad(p=1,\dots,n-\nu+1);
 \end{equation}
where $T$ is a $p\times(p+\nu)$-matrix of the form
  \begin{equation} \label{2.9T}
T=        \begin{pmatrix}
                    &   & s_{\nu-1} &\ldots & s_{p+\nu-2} \\
                    &   &           & \ddots & \vdots \\
             \bf 0  &   &           &         & s_{\nu-1} \\
        \end{pmatrix}\quad(p=1,\dots,n-\nu+1);
 \end{equation}
The indices $\nu_\pm({\cS}_{p})$, $\nu_0({\cS}_{p})$ and the normal
indices ${\mathfrak n}_j$ of the sequence $({\mathfrak
s}_i)_{i=0}^{\ell-2\nu}$ are given by
   \begin{equation} \label{2.9a}
\nu_\pm({\cS}_{p})=\nu_\pm(S_{p+\nu})-\nu_\pm(S_{\nu})\quad(p=1,\dots,n-\nu+1);
 \end{equation}
 \begin{equation} \label{2.9b}
    \nu_0({\cS}_{p})=\nu_0(S_{p+\nu})\quad(p=1,\dots,n-\nu+1),
  \end{equation}
\begin{equation}\label{eq:NIwidehat}
    {\mathfrak n}_j=n_{j+1}-\nu\quad(j=1,\dots,N-1).
\end{equation}
\end{enumerate}
\end{lemma}
Let us define the following polynomial $m$ by
\begin{equation}\label{3p.3.8.m.pol}
    m(z)=\frac{a(z)-a(0)}{bz}\qquad (\textup{deg}(m)=\nu_{}-1).
\end{equation}
Due to~\eqref{eq:2.5a}, $m(z)$ takes the form
\begin{equation}\label{2p.3.10}
        m(z)=\frac{(-1)^{\nu+1}}{D_{\nu}}
        \begin{vmatrix}
            0 & \ldots & 0 & s_{\nu-1} & s_{\nu} \\
            \vdots &  & \ldots & \ldots & \vdots \\
            s_{\nu-1} & \ldots & \ldots & \ldots & s_{2\nu_{}-2} \\
            1 & z & \ldots& z^{\nu-2}& z^{\nu-1} \\
        \end{vmatrix}\quad (D_\nu:=\det S_\nu).
\end{equation}
and the leading coefficient of $m$ is calculated by
\begin{equation}\label{2p.3.11}
    (-1)^{\nu+1}\frac{D^{+}_{\nu-1}}{D_{\nu}}=\frac{1}{s_{\nu-1}}.
\end{equation}

Let us reformulate Lemma~\ref{lem:D03} in terms of the polynomial $m$.
\begin{lemma}\label{lem:D02}
Let a real sequence $\textbf{s}=\{s_{j}\}_{j=0}^{\ell}$ satisfy the
conditions~\eqref{3p.3.1} ($\ell\ge 2\nu-1$), let
$\cN(\textbf{s})=\{n_j\}_{j=1}^N$ and let the polynomial
$m(z)=\sum_{j=0}^{\nu-1} m_j z^j$ be defined by \eqref{2p.3.10}.
Then a function $f$ (meromorphic on $\dC\setminus\dR$) admits the
asymptotic expansion~\eqref{eq:2.4} if and only if the function
$-1/f(z)$ admits the asymptotic expansion
\begin{equation} \label{eq:2.4b}
-1/f(z) = zm(z)+g(z),\quad
z \wh \to \infty,
\end{equation}
where $g({z})$ satisfies one of
the following conditions:
\begin{enumerate}\def\labelenumi{(\roman{enumi})}
\item if $\ell=2\nu-2$ then $g({z})=o({z})$, ${z} \wh \to
\infty$;

\item if $\ell\ge 2\nu-1$ then $g({z})$ has the asymptotic expansion
\begin{equation}\label{eq:2.4D}
     g({z})=-{\mathfrak s}_{-1}
 -\frac{{\mathfrak s}_{0}}{{z}}-\dots-\frac{{\mathfrak s}_{\ell}}{{z}^{\ell+1}}
  +o\left(\frac{1}{{z}^{\ell+1}}\right),\quad {z} \wh \to
\infty,
\end{equation}
where the sequence $({\mathfrak s}_i)_{i=0}^{\ell-2\nu}$ is
determined by the matrix equation
\begin{equation}\label{eq:2.4GG}
     T(m_{\nu-1},\ldots,m_0,-{\mathfrak s}_{-1},\ldots,-
 {\mathfrak s}_{\ell-2\nu}) \, T(s_{\nu-1},\dots,s_\ell) = I_{\ell-\nu+2}.
\end{equation}
The indices $\nu_\pm({\cS}_{p})$, $\nu_0({\cS}_{p})$ and the normal
indices ${\mathfrak n}_j$ of the sequence $({\mathfrak
s}_i)_{i=0}^{\ell-2\nu}$ are given
by~\eqref{2.9a}~--~\eqref{eq:NIwidehat}.
\end{enumerate}
\end{lemma}

\begin{remark}\label{rem:2.5}
It follows from the equality~\eqref{eq:2.4G} and \cite[Proposition~2.1]{Der03} that the sequence
$\left\{{\mathfrak s}_{i}\right\}_{i=-1}^{\ell-2\nu_{1}}$
can be found by the equalities
 \begin{equation}\label{eq:s_{-1}}
 {\mathfrak s}_{-1}=\frac{(-1)}{s_{\nu_{1}-1}}^{\nu_{1}+1}\frac{D_{\nu_{1}}^+}{D_{\nu_{1}}},
 \end{equation}
  \begin{equation}\label{eq:si^{1}}
 {\mathfrak s}_{i}=
\frac{(-1)^{i+\nu_{1}}}{s_{\nu_{1}-1}^{i+\nu_{1}+2}}\begin{vmatrix}
                                            s_{\nu_{1}} & s_{\nu_{1}-1} & 0 & \ldots & 0 \\
                                            \vdots & \ddots & \ddots & \ddots & \vdots \\
                                            \vdots &   & \ddots & \ddots & 0 \\
                                            \vdots &   &   & \ddots & s_{\nu_{1}-1} \\
                                            s_{2\nu_{1}+i} &  \ldots &  \ldots    &  \ldots & s_{\nu_{1}} \\
                                          \end{vmatrix}\quad
                                          i=\overline{0,\ell-2\nu_{1}}.
 \end{equation}
\end{remark}
Next statement is an analog  of Lemma~\ref{lem:D02} which is applicable for expansions containing constants.
\begin{lemma}\label{lem:D03}
Let  $\textbf{s}=\{{\mathfrak s}_{j}\}_{j=-1}^{\ell}$ be a real
sequence such that ${\mathfrak s}_{-1}\ne 0$. Let
$\cN(\textbf{s})=\{n_j\}_{j=1}^N$, $n=[\ell/2]$ and let
 $l=1/{\mathfrak s}_{-1}$. Then a function $g$ (meromorphic
on $\dC\setminus\dR$) admits the asymptotic
expansion~\eqref{eq:2.4D} if and only if the function $-1/g(z)$
admits the representation
\begin{equation} \label{eq:2.4l}
-1/g(z) = l+f(z),\quad
z \wh \to \infty,
\end{equation}
where $f({z})$ satisfies one of
the following conditions:
\begin{enumerate}\def\labelenumi{(\roman{enumi})}
\item if $\ell=-1$ then $f({z})=o({1})$, ${z} \wh \to
\infty$;

\item if $\ell\ge0$ then $f({z})$ has the asymptotic expansion
\begin{equation} \label{eq:2.4EE}
 f({z})=
 -\frac{s_{0}}{{z}}-\dots-\frac{s_{\ell}}{{z}^{\ell+1}}
  +o\left(\frac{1}{{z}^{\ell+1}}\right), \quad {z} \wh \to \infty,
\end{equation}
where the sequence $({\mathfrak s}_i)_{i=0}^{\ell-2\nu}$ is
determined by the matrix equation
\begin{equation}\label{eq:2.4G2}
     T({\mathfrak s}_{-1},\dots,{\mathfrak s}_\ell)
     T(l,-s_{0},\ldots,- s_{\ell}) \,  = I_{\ell+2}.
\end{equation}
The indices $\nu_\pm( S_{p})$, $\nu_0( S_{p})$  are given by
   \begin{equation} \label{eq:2.9a}
\nu_0( {S_{p}})=\nu_0({\cS}_{p}),\quad \nu_\pm(
{S_{p}})=\nu_\pm({\cS}_{p})\quad(p=0,\dots,n+1).
 \end{equation}
\end{enumerate}
\end{lemma}
\begin{proof}
If $\ell=-1$, then \eqref{eq:2.4} takes the form
\[
g(z)=-{\mathfrak s}_{-1}+o(1),\quad {z} \wh \to
\infty,
\]
and hence (i) is clear.

Assume that $\ell\ge 0$. Then by Lemma~\ref{lem:D01} one obtains the
representation \eqref{eq:2.4a}, \eqref{eq:2.4E} for $-1/g$ with
coefficients $s_j$ $(j=0,\dots,\ell)$, satisfying \eqref{eq:2.4G2}.
Multiplying \eqref{eq:2.4G2} with $\ell$ replaced by $2n$
$(n=[\ell/2])$ both from the left and from the right by the matrix
$J_{2n+2}$ one obtains the equality $AB=I_{n+2}$, or in the block
form
   \begin{equation} \label{2.10}
\begin{pmatrix}
   0_{(n+1)\times(n+1)}  &   A_{12} \\
    A_{12}^*            &  A_{22}\\
\end{pmatrix}
\begin{pmatrix}
   B_{11}   &  B_{12} \\
   B_{12}^* &  0_{(n+1)\times(n+1)}\\
\end{pmatrix}
=I_{2n+2},
   \end{equation}
where
\[
 A_{12}=\begin{pmatrix}
   0  & \dots       &  {\mathfrak s}_{-1} \\
\vdots&     \ddots  &   \vdots \\
{\mathfrak s}_{-1} & \dots      &   {\mathfrak s}_{n-1} \\
\end{pmatrix}
\in\dC^{(n+1)\times(n+1)},
\]
\begin{equation} \label{2.11}
 A_{22}={\cS}_{n+1}\in\dC^{(n+1)\times(n+1)},\quad B_{11}=-J_{n+1}
S_{n+1}J_{n+1}\in \dC^{(n+1)\times(n+1)}
   \end{equation}
and $B_{12}$, $B_{12}^*$ are some matrices from
$\dC^{(n+1)\times(n+1)} $. Notice that the matrix $A$ is invertible.
If in addition, the matrix $A_{22}$ is invertible then its Schur
complement
\[
B_{11}^{-1}=-A_{12}A_{22}^{-1}A_{12}^*
\]
and hence the matrix $B_{11}=(A_{11})^{-1}$ is also invertible. In
view of~\eqref{2.11} this implies that the matrix $S_{n+1}$ is
invertible. The converse is also true by similar arguments. This
proves the equalities~\eqref{eq:2.9a}.
\end{proof}

For a sequence ${\mathbf s}=\{s_i\}_{i-=1}^{2n-1}$ let us set
\begin{equation}\label{eq:S_n-}
    S^-_{n}= \begin{pmatrix}
                     s_{-1} &\cdots & s_{n-2} \\
                     \cdots & \cdots& \cdots \\
                    s_{n-2} & \cdots & s_{2n-3}\\
                   \end{pmatrix}\quad(n\in\dN).
\end{equation}
\begin{corollary}  \label{cor:3.5}
Under the assumptions of Lemma~\ref{lem:D02} the indices $\nu_0({\cS}_{p}^-)$ and $\nu_\pm({\cS}_{p}^-)$
for matrices ${\cS}_{p}^-=({\mathfrak s}_{i+j-1})_{i,j=0}^{p-1}$  are evaluated by the equalities
 \begin{equation} \label{2.9bb}
    \nu_0({\cS}_{p}^-)=\nu_0(S_{p+\nu-1}^+)\quad(p=1,\dots,n-\nu+1,\,n=[\ell/2]);
  \end{equation}
    \begin{equation} \label{2.9b_0}
\nu_\pm({\cS}_{p}^-)=\nu_\pm(S_{p+\nu-1}^+)-\nu_\pm(S_{\nu-1}^+)\quad\mbox{if }s_0=0\quad(p=1,\dots,n-\nu+1);
 \end{equation}
   \begin{equation} \label{2.9b_pm}
\nu_\pm({\cS}_{p}^-)=\nu_\pm(S_{p}^+)\quad\mbox{if }s_0\ne 0\quad(p=1,\dots,n).
 \end{equation}
\end{corollary}
\begin{proof}
Assume that $s_0=0$. Then it follows from~\eqref{eq:2.4a}, \eqref{eq:2.4E} that
\begin{equation}\label{eq:2.4cor}
     zf({z})=
 -\frac{s_{\nu-1}}{{z}^{\nu-1}}-\dots-\frac{s_{2i-1}}{{z}^{2i-1}}
  +o\left(\frac{1}{{z}^{2i-1}}\right), \quad {z} \wh \to \infty,
\end{equation}
\begin{equation} \label{eq:2.4Ecor}
 -\frac{1}{zf(z)}=m(z)-\frac{{\mathfrak s}_{-1}}{z}
-\dots-\frac{{\mathfrak s}_{2i-1-2\nu}}{{z}^{2i-2\nu}}
  +o\left(\frac{1}{{z}^{2i-2\nu}}\right), \quad {z} \wh \to \infty,
\end{equation}
Applying Lemma~\ref{lem:D02} to $zf(z)$ and using the expansions~\eqref{eq:2.4cor} and~\eqref{eq:2.4Ecor} one obtains
\[
   \nu_0({\cS}_{p-\nu+1}^-)=\nu_0(S_{p}^+)\quad(p=\nu,\dots,\left[\ell/2\right]).
\]
\[
   \nu_\pm({\cS}_{p-\nu+1}^-)=\nu_\pm(S_{p}^+)-\nu_\pm(S_{\nu-1}^+)\quad
   (p=\nu,\dots,\left[\ell/2\right]).
\]

If $s_0\ne 0$, then $\nu=1$ and the expansions~\eqref{eq:2.4cor} and \eqref{eq:2.4Ecor} take the form
\[
     zf({z})=-s_0
 -\frac{s_{1}}{{z}}-\dots-\frac{s_{2i-1}}{{z}^{2i-1}}
  +o\left(\frac{1}{{z}^{2i-1}}\right), \quad {z} \wh \to \infty,
\]
\[
 -\frac{1}{zf(z)}=m-\frac{{\mathfrak s}_{-1}}{z}
-\dots-\frac{{\mathfrak s}_{2i-3}}{{z}^{2i-2}}
  +o\left(\frac{1}{{z}^{2i-2}}\right), \quad {z} \wh \to \infty,
\]
where $m=1/s_0$ and by Lemma~\ref{lem:D03}
\[
   \nu_0({\cS}_{p}^-)=\nu_0(S_{p}^+),\quad
   \nu_\pm({\cS}_{p}^-)=\nu_\pm(S_{p}^+)\quad
   (p=1,\dots,\left[\ell/2\right]).
\]
This proves~\eqref{2.9b_0}-\eqref{2.9b_pm}.
\end{proof}
\begin{corollary}  \label{cor:3.6}
Under the assumptions of Lemma~\ref{lem:D03} the indices $\nu_0({S}_{p}^+)$ and $\nu_-({S}_{p}^+)$ for matrices ${S}_{p}^+=({\mathfrak s}_{i+j-1})_{i,j=0}^{p-1}$  are evaluated by the equalities
\begin{equation}
    \begin{split} \label{eq:2.9b}
        &\nu_0( {S_{p}}^+)=\nu_0({\cS}_{p+1}^-) \quad
        (p=1,\dots,n+1);\\& \nu_-( S_{p})=\nu_-({\cS}_{p+1}^-),\quad\mbox{if }{\mathfrak s}_{-1}>0\quad(p=1,\dots,n+1);\\&
        \nu_-( S_{p})=\nu_-({\cS}_{p+1}^-)-1,\quad\mbox{if }{\mathfrak s}_{-1}<0\quad(p=1,\dots,n+1).
    \end{split}
\end{equation}
\end{corollary}
\begin{proof}
Lemma~\ref{lem:D02} applied to the asymptotic expansions
\begin{equation}\label{eq:2.4bg}
     \frac{g({z})}{z}=- \frac{{\mathfrak s}_{-1}}{z}
 -\frac{{\mathfrak s}_{0}}{{z^2}}-\dots-\frac{{\mathfrak s}_{\ell}}{{z}^{\ell+2}}
  +o\left(\frac{1}{{z}^{\ell+1}}\right), \quad {z} \wh \to \infty,
\end{equation}
\begin{equation} \label{eq:2.4Ez}
- \frac{z}{g({z})}=lz
 -s_{0}-\frac{s_{1}}{{z}}-\dots-\frac{s_{\ell}}{{z}^{\ell}}
  +o\left(\frac{1}{{z}^{\ell}}\right), \quad {z} \wh \to \infty,
\end{equation}
where  $l=1/{\mathfrak s}_{-1}$, gives
  \begin{equation} \label{eq:2.9c}
   \begin{split}
\nu_0( {S_{p}}^+)&=\nu_0({\cS}_{p+1}^-)\quad(p=1,\dots,n);\\
\nu_-( S_{p}^+)&=\nu_-({\cS}_{p+1}^-)-\nu_-({\cS}_1^-)\quad(p=1,\dots,n);\\
    \end{split}
 \end{equation}
Now the equalities~\eqref{eq:2.9b} are implied by~\eqref{eq:2.9c} since ${\cS}_1^-=({\mathfrak s}_{-1})$.
\end{proof}

\subsection{Class $\cU_\kappa(J)$ and linear fractional transformations}
Let $\kappa_1\in\dN$ and let $J$ be a  $2\times 2$ signature matrix
\[
J=\begin{pmatrix}
   0  &  -i \\
   i  &  0\\
\end{pmatrix}.
\]
A $2\times 2$ matrix valued function $W(z)=(w_{i,j}(z))_{i,j=1}^2$
that is meromorphic in $\dC_+$ belongs to the class ${\cU}_\kappa(J)$ of {\it generalized $J$-inner} matrix valued functions
if:
\begin{enumerate}
\item[(i)]
the kernel
\begin{equation}\label{kerK}
{\mathsf K}_\omega^W(z)=
\frac{J-W(z)JW(\omega)^*}{-i(z-\bar \omega)}
\end{equation}
has $\kappa$ negative squares in ${\mathfrak H}_W^+\times{\mathfrak
H}_W^+$  and
\item[(ii)]
$J-W(\mu)JW(\mu)^*=0$ for a.e.  $\mu\in\dR$,
\end{enumerate}
where ${\mathfrak H}_W^+$ denotes the
domain of holomorphy of $W$ in $\dC_+$.

Consider the linear fractional transformation
\begin{equation}\label{eq:0.9}
    T_W[\tau]=(w_{11}\tau(z)+w_{12})(w_{21}\tau(z)+w_{22})^{-1}
\end{equation}
associated with the matrix valued function $W(z)$. The linear fractional transformation associated with the product $W_1W_2$ of two matrix valued function $W_1(z)$ and $W_2(z)$, coincides with the composition $T_{W_1}\circ T_{W_2}$.

As is known, if $W\in{\cU}_{\kappa_1}(J)$ and $\tau\in{\mathbf N}_{\kappa_2}$ then $T_W[\tau]\in{\mathbf N}_{\kappa'}$, where $\kappa'\le{\kappa_1+\kappa_2}$.
In the present paper two partial cases, in which  the preceding inequality becomes equality, will be needed.
\begin{lemma}\label{lem:W_m}
Let $m(z)$ be a real polynomial such that $\kappa_-(zm)=\kappa_1$, $\kappa_-(m)=k_1$, let $M$ be a $2\times 2$ matrix valued function
\begin{equation}\label{eq:W_m}
    M(z)=\begin{pmatrix}
   1  &  0 \\
   -z m(z)  &  1\\
\end{pmatrix}
\end{equation}
and let $\tau$ be a meromorphic function, such that
$\tau(z)^{-1}=o(z)\mbox{ as }z\wh\to\infty$.
Then the following equivalences hold:
\begin{equation}\label{eq:w_m_k}
    \tau\in{\mathbf N}_{\kappa_2}\Longleftrightarrow T_M[\tau]\in {\mathbf N}_{\kappa_1+\kappa_2},
\end{equation}
\begin{equation}\label{eq:w_m_kk}
    \tau\in{\mathbf N}_{\kappa_2}^{k_2}\Longleftrightarrow  T_M[\tau]\in {\mathbf N}_{\kappa_1+\kappa_2}^{k_1+k_2}.
\end{equation}
\end{lemma}
\begin{proof}
Let us set $f=T_M[\tau]$. Then
\begin{equation}\label{eq:fT_M}
    -\frac{1}{f(z)}=zm(z)-\frac{1}{\tau(z)}.
\end{equation}
It follows from~\eqref{eq:fT_M} and Proposition~\ref{prop:2.1} (3) that $-\frac{1}{f}\in {\mathbf N}_{\kappa_1+\kappa_2}$. In view of Proposition~\ref{prop:2.1} (1) this implies~\eqref{eq:w_m_k}.

Dividing~\eqref{eq:fT_M} by $z$ one obtains
\begin{equation}\label{eq:fT_M2}
    -\frac{1}{zf(z)}=m(z)-\frac{1}{z\tau(z)}.
\end{equation}
Since $(z\tau(z))^{-1}=o(1)\mbox{ as }z\wh\to\infty$, then by Proposition~\ref{prop:2.1} (3) $-\frac{1}{zf}\in {\mathbf N}_{k_1+k_2}$ and hence  $zf\in {\mathbf N}_{k_1+k_2}$.  This proves~\eqref{eq:w_m_kk}.
\end{proof}

\begin{lemma}\label{lem:W_l}
Let $l(z)$ be a real polynomial such that $\kappa_-(l)=\kappa_1$, $\kappa_-(zl(z))=k_1$, let $L(z)$ be a $2\times 2$ matrix valued function
\begin{equation}\label{eq:W_l}
    L(z)=\begin{pmatrix}
   1  &  l(z) \\
   0  &  1\\
\end{pmatrix}
\end{equation}
and let $\tau$ be a meromorphic function, such that
$\tau(z)^{-1}=o(1)\mbox{ as }z\wh\to\infty$.
Then the following equivalences hold:
\begin{equation}\label{eq:w_l_k}
    \tau\in{\mathbf N}_{\kappa_2}\Longleftrightarrow T_L[\tau]\in {\mathbf N}_{\kappa_1+\kappa_2},
\end{equation}
\begin{equation}\label{eq:w_l_kk}
    \tau\in{\mathbf N}_{\kappa_2}^{k_2}\Longleftrightarrow T_L[\tau]\in {\mathbf N}_{\kappa_1+\kappa_2}^{k_1+k_2}.
\end{equation}
\end{lemma}
\begin{proof}
Let us set $f=T_L[\tau]$. Then~\eqref{eq:w_l_k} is implied by the equality
\begin{equation}\label{eq:fT_L}
    {f(z)}=l(z)+\tau(z).
\end{equation}
and Proposition~\ref{prop:2.1} (3).
Multiplying~\eqref{eq:fT_L} by $z$ one obtains
\begin{equation}\label{eq:fT_L2}
   {zf(z)}=zl(z)+z\tau(z).
\end{equation}
Since $z\tau(z)=o(z)\mbox{ as }z\wh\to\infty$, then by Proposition~\ref{prop:2.1} (3) ${zf}\in {\mathbf N}_{k_1+k_2}$. This proves~\eqref{eq:w_l_kk}.
\end{proof}


\section{Basic moment problem in  $\mathbf{N}_{\kappa}^{ k}$ }
In this section we consider a basic moment problem in  Nevanlinna
class $\mathbf{N}_{\kappa}^{ k}$ and describe its solutions. {Odd} and even moment problems will be treated separately.
In both cases one step of
the Schur algorithm  will be considered.

\subsection {Basic {odd} moment problem $MP_{\kappa}^{k}(\textbf{s},2\nu_{1}-2)$}
An {odd} moment problem $MP_{\kappa}^{k}(\textbf{s},2n-2)$ is called
nondegenerate if
\begin{equation}\label{eq:even}
    D_n\ne 0 \quad\mbox{ and }\quad D_{n-1}^+\ne 0.
\end{equation}
By definition \eqref{eq:3p.nu} this means that $n\in\cN({\bf s})$.
A nondegenerate {odd} moment problem $MP_{\kappa}^{k}(\textbf{s},2n-2)$ will be called basic,
if $n$ is the only normal index of ${\bf s}$, i.e. $n=\nu_1$ and  $\cN(\textbf{s})=\{\nu_1\}$.
This case can be characterized by the conditions~\eqref{3p.3.1}.

The basic moment problem $MP_{\kappa}^{k}(\textbf{s},2\nu_{1}-2)$  can be reformulated as follow:\\
Given a sequence $\textbf{s}=\{s_{j}\}_{j=0}^{2\nu_{1}-2}$ with $\cN(\textbf{s})=\{\nu_1\}$,
find all functions $f\in N_{\kappa}^{k}$ such that
\begin{equation}\label{3p.3.3}
    f(z)=-\frac{s_{\nu_{1}-1}}{z^{\nu_{1}}}-\cdots
    -\frac{s_{2\nu_{1}-2}}{z^{2\nu_{1}-1}}+o\left(\frac{1}{z^{2\nu_{1}-1}}\right), \quad {z} \wh \to \infty.
\end{equation}

Let
 $\textbf{s}=\{s_{j}\}_{j=0}^{2\nu_{1}-2}$ be a  sequence of
real numbers  from ${\cH}$ and let~\eqref{3p.3.1} holds. Then
$\textbf{s} \in{\cH}_{\kappa_1, 2\nu_{1}-2}^{k_1}$, where
$\kappa_1$ and $k_1$ are defined by
\begin{equation}
\label{3p.3.4}
    \kappa_{1}=\nu_-(S_{\nu_{1}})=
    \left\{
    \begin{array}{cl}
       \left[\frac{\nu_{1}+1}{2}\right],&       \mbox{ if }\nu_{1}\mbox{ is odd and } s_{\nu_{1}-1}<0;\\
       \left[\frac{\nu_{1}}{2}\right],&  \mbox{ otherwise} .\\
    \end{array}
    \right.
\end{equation}

\begin{equation}\label{3p.3.5}
    k_{1}=\nu_{-}(S_{\nu_{1}-1}^+)
    =\left\{
    \begin{array}{ccl}
        [\frac{\nu_{1}}{2}],&        &\mbox{ if }\nu_{1} \mbox{ is {even} and } s_{\nu_{1}-1}<0;\\

        [\frac{\nu_{1}-1}{2}],&  &\mbox{ otherwise} .\\
    \end{array}
    \right.
\end{equation}

It follows from (\ref{3p.3.4}) and (\ref{3p.3.5}), that
\begin{equation}\label{3p.3.6}
    k_{1}=\nu_{-}(S_{\nu_{1}-1}^+)=\left\{
    \begin{array}{rcl}
    \kappa_{1}-1,&        &\mbox{ if }\nu_{1} \mbox{ is odd and } s_{\nu_{1}-1}<0;\\
    \kappa_{1}-1,&        &\mbox{ if }\nu_{1} \mbox{ is even and } s_{\nu_{1}-1}>0;\\
    \kappa_{1},&  &\mbox{
    otherwise} .\\
    \end{array}
    \right.
\end{equation}

Let  $m_{1}(z)$ be the polynomial defined by~\eqref{2p.3.10} with
$\nu=\nu_1$. Then it follows from \eqref{3p.kappaP}
and~\eqref{3p.3.4}, \eqref{3p.3.5}, that
\begin{equation}\label{eq:kappa_1k_1}
    \kappa_{1}
    =\kappa_-(zm_1),\quad
    k_1
    =\kappa_-(m_1).
\end{equation}

\begin{lemma}\label{lem:3.1}
Let $\nu_{1}$ be the first normal index of the sequence $\textbf{s}=\{s_{j}\}_{j=0}^{2\nu_{1}-2}$, let
polynomial $m_{1}$ be defined by~\eqref{2p.3.10}
and let $f\in \mathbf{N}_{\kappa}$  have the asymptotic
expansion~\eqref{3p.3.3}.
 Then $f$ admits the following representation
\begin{equation}\label{eq:3.1}
    f(z)=-\frac{1}{zm_{1}(z)+{g}(z)},
\end{equation}
where
\begin{equation}\label{eq:3.1B}
   {g}\in \mathbf{N}_{\kappa-\kappa_{1}}\quad\mbox{ and }\quad
{g}(z)=o(z),\quad\quad z\widehat{\rightarrow}\infty.
\end{equation}
Conversely, if ${g}$ satisfies \eqref{eq:3.1B} and $f$ is defined
by~\eqref{eq:3.1}, then $f\in \mathbf{N}_{\kappa}$.
\end{lemma}
\begin{proof}
By Lemma~\ref{lem:D02}, $f$ admits the
representation~\eqref{3p.3.3}, where ${g}(z)=o(z)\mbox{ as
}z\widehat{\rightarrow}\infty$. Next, since $f\in
\mathbf{N}_{\kappa}$ then also $-1/f\in \mathbf{N}_{\kappa}$ and
then it follows from the equality
\begin{equation}\label{eq:3.1A}
    -1/f(z)=zm_{1}(z)+g(z)
\end{equation}
and Proposition~\ref{prop:2.1} (3) that ${g}\in
\mathbf{N}_{\kappa-\kappa_{-}(zm_1)}$. Since
by~\eqref{eq:kappa_1k_1} $\kappa_{-}(zm_1)=\kappa_1$ one gets
${g}\in \mathbf{N}_{\kappa-\kappa_{1}}$.

Conversely, if ${g}$ satisfies \eqref{eq:3.1B} then by
Lemma~\ref{lem:D02} $f$  has the asymptotic expansion~\eqref{3p.3.3}
and by~\eqref{eq:3.1A} and Proposition~\ref{prop:2.1} (3) $f\in
\mathbf{N}_{\kappa_1+(\kappa-\kappa_{1})}=\mathbf{N}_{\kappa}$.
\end{proof}

\begin{remark} \label{rem:3.2}
Replacing $g$ by $-1/g_1$ in~\eqref{eq:3.1}, we can rewrite it as
follows
\begin{equation}\label{3p.3.8'}
    f(z)=
    T_{M_{1}}[g_1]=\frac{g_1(z)}{-zm_{1}(z)g_1(z)+1},
\end{equation}
where the polynomial $m_{1}(z)$ is defined by~\eqref{2p.3.10}, and the matrix valued function
\begin{equation}\label{eq:M_1}
        M_{1}(z)=\begin{pmatrix}
                     1 & 0 \\
            -zm_{1}(z) & 1 \\
    \end{pmatrix}
\end{equation}
belongs to the class $\cU_{\kappa_1}(J)$. The statement of Lemma~\ref{lem:3.1} can be reformulated as follows
\begin{equation}\label{eq:TM1}
T_{M_1}[g_1]\in {\mathbf N}_{\kappa}\Longleftrightarrow g_1\in {\mathbf N}_{\kappa-\kappa_{1}}\quad\&\quad
 \frac{1}{g_1(z)}=o(z),\quad
z\widehat{\rightarrow}\infty.
\end{equation}
Moreover, it follows from Lemma~\ref{lem:W_m} that
\begin{equation}\label{eq:TM1_k1}
    T_{M_1}[g_1]\in {\mathbf N}_{\kappa}^k\Longleftrightarrow g_1\in {\mathbf N}_{\kappa-\kappa_{1}}^{k-k_{1}}\quad\&\quad
     \frac{1}{g_1(z)}=o(z),\quad
    z\widehat{\rightarrow}\infty.
\end{equation}
In fact, the reason for switching to reciprocal function $g_1$ is motivated by~\eqref{eq:TM1_k1},  it helps to keep $g_1$ staying in a generalized Stieltjes class ${\mathbf N}_{\kappa-\kappa_{1}}^{k-k_{1}}$.
\end{remark}

Combining Lemma~\ref{lem:3.1} and Remark~\ref{rem:3.2} with calculations in~\eqref{eq:kappa_1k_1} one obtains

\begin{theorem}\label{3p.th3.3}
Let $\nu_{1}$ be the first normal index of the sequence $\textbf{s}=\{s_{i}\}_{i=0}^{2\nu_{1}-2}$, let
$m_{1}$, $\kappa_1$ and $k_1$ be defined by~\eqref{2p.3.10}, \eqref{3p.3.4} and by~\eqref{3p.3.5}, respectively,
and let $\ell\ge 2\nu_1-2$. Then:

\begin{enumerate}
  \item [(1)] The problem $ {MP}_{\kappa}^{k}(\textbf{s},
      \ell)$ is solvable if and only if
\begin{equation}\label{eq:Basic_Solv}
    \kappa_1\le\kappa\quad\mbox{and}\quad k_1\le k.
\end{equation}
   \item [(2)]$f\in \mathcal{M}_{\kappa}^{k}(\textbf{s},
  2\nu_{1}-2)$ if and only if $f$ admits the
  representation
\begin{equation}\label{eq:LFT_M1}
    f=T_{M_1}[\tau],
\end{equation}
   where
  $\tau$ satisfies the conditions
\begin{equation}\label{eq:tau_1}
    \tau \in
    {\mathbf N}_{\kappa-\kappa_{1}}^{k-k_{1}}\quad\mbox{and}\quad
    \frac{\displaystyle1}{\displaystyle\tau(z)}=o(z),\quad
    z\widehat{\rightarrow}\infty.
\end{equation}

\item [(3)]  If $\ell> 2\nu_{1}-2$, then $f\in \mathcal{M}_{\kappa}^{k}(\textbf{s},
  \ell)$  if and only if $f$ admits the
  representation $f=T_{M_1}[g_1]$,
  where $g_1\in
  {\mathbf N}_{\kappa-\kappa_{1}}^{k-k_{1}}$ and
  $-\frac{\displaystyle1}{\displaystyle g_1(z)}$ has the following
  asymptotic expansion
\begin{equation}\label{3p.3.19'}
    -\frac{\displaystyle1}{\displaystyle g_1(z)}
    =-{\mathfrak s}_{-1}-\frac{{\mathfrak s}_{0}}{z}-
    \cdots-\frac{{\mathfrak s}_{n-2\nu_{1}}}{z^{n-2\nu_{1}+1}}
    +o\left(\frac{1}{z^{n-2\nu_{1}+1}}\right),\quad
    z\widehat{\rightarrow}\infty,
\end{equation}
and the sequence
$\left\{{\mathfrak s}_{i}\right\}_{i=-1}^{n-2\nu_{1}}$
is
determined by the matrix equation
\begin{equation}\label{eq:2.4G^1}
     T(m_{\nu_1-1}^{(1)},\ldots,m_0^{(1)},-{\mathfrak s}^{(1)}_{-1},\ldots,-
 {\mathfrak s}^{(1)}_{\ell-2\nu_1}) \, T(s_{\nu_1-1},\dots,s_\ell) = I_{\ell-\nu_1+2}.
\end{equation}
\end{enumerate}
 \end{theorem}
\begin{proof} $(1)$ Assume that $f\in \mathcal{M}_{\kappa}^{k}(\textbf{s},
      \ell)$. The inequality $\kappa_1\le\kappa$ is implied by Proposition~\ref{prop:2.1} (4). Next, since $zf\in {\mathbf N}_{k}$ and
\begin{equation}\label{eq:3p.f(1)asymp}
zf(z)+s_0=-\frac{s_{1}}{z}-\frac{s_{2}}{z^{2}}-\cdots-\frac{s_{\ell}}
{z^{\ell}}+o\left(\frac{1}{z^{\ell}}\right),\qquad
z\widehat{\rightarrow}\infty,
\end{equation}
then necessarily, by Corollary~\ref{cor:munu} (4)
$k_1=\nu_-(S_{\nu-1}^+)\le k$.

 $(2)$
Assume $f$ belongs to $ N_{\kappa}^{k}$ and has the asymptotic
expansion~\eqref{3p.3.3}.
Then by Lemma~\ref{lem:3.1} and Remark~\ref{rem:3.2}, the function
$f\in\mathcal{M}_{\kappa}^{k}(\textbf{s}, 2\nu_{1}-2)$ has the
representation~\eqref{3p.3.8'} if and only if~\eqref{eq:tau_1}
holds.

$(3)$ Suppose $f$ belongs to $ \mathcal{M}_{\kappa}^{k}(\textbf{s},
  \ell)$.
By Lemma~\ref{lem:D02} and Remark~\ref{rem:3.2}, the function $f$ admits the representation $f=T_{M_1}[g_1]$,
where
$g_1$ satisfies~\eqref{3p.3.19'} and the sequence
$\left\{{\mathfrak s}_{i}^{(1)}\right\}_{i=-1}^{n-2\nu_{1}}$
is
determined by~\eqref{eq:2.4G^1}. Moreover,
$g_1\in
N_{\kappa-\kappa_{1}}^{k-k_{1}}$ by Lemma~\ref{lem:W_m}.

The converse also follows from Lemma~\ref{lem:D02} and Lemma~\ref{lem:W_m}.
\end{proof}
\begin{remark}\label{rem:2.5A}
It follows from the equality~\eqref{eq:2.4G} and
\cite[Proposition~2.1]{Der03} that the sequence
$\left\{{\mathfrak s}_{i}^{(1)}\right\}_{i=-1}^{\ell-2\nu_{1}}$
can be found by the equalities
 \begin{equation}\label{eq:s^{1}}
 {\mathfrak s}_{-1}^{(1)}=\frac{(-1)^{\nu_{1}+1}}{s_{\nu_{1}-1}}\frac{D_{\nu_{1}}^+}{D_{\nu_{1}}},
 \end{equation}
  \begin{equation}\label{eq:si^{2}}
     {\mathfrak s}_{i}^{(1)}=
     \frac{(-1)^{i+\nu_{1}}}{s_{\nu_{1}-1}^{i+\nu_{1}+2}}
        \begin{vmatrix}
         s_{\nu_{1}} & s_{\nu_{1}-1} & 0 & \ldots & 0 \\
         \vdots & \ddots & \ddots & \ddots & \vdots \\
         \vdots &   & \ddots & \ddots & 0 \\
         \vdots &   &   & \ddots & s_{\nu_{1}-1} \\
         s_{2\nu_{1}+i} &  \ldots &  \ldots    &  \ldots & s_{\nu_{1}} \\
     \end{vmatrix}
     \quad i=\overline{0,\ell-2\nu_{1}}.
 \end{equation}
\end{remark}

\subsection{Basic {even} moment problem $MP_{\kappa}^{k}(\textbf{s},2\mu_{1}-1)$}
An {even} moment problem \\
$MP_{\kappa}^{k}(\textbf{s},2n-1)$ is called nondegenerate, if
\begin{equation}\label{eq:odd}
    D_n\ne 0 \quad\mbox{ and }\quad D_{n}^+\ne 0.
\end{equation}
By classification~\eqref{eq:3p.nu}, \eqref{eq:3p.mu} this means that
$n\in\cN({\bf s})$ and $n=\mu_j$ for some $j$. A nondegenerate
{even} moment problem $MP_{\kappa}^{k}(\textbf{s},2n-2)$ will be
called basic, if $n$ is the smallest index such that~\eqref{eq:odd}
holds.  Therefore, the basic {even} moment problem coincides with
the  problem $MP_{\kappa}^{k}(\textbf{s},2\mu_{1}-1)$. Regarding to
the conditions $\nu_1=\mu_1$ or $\nu_1<\mu_1$ the set of normal
indices consists either of one element $\nu_1$ or of two elements
$\nu_1$ and $\mu_1$.

The basic {even} moment problem
$MP_{\kappa}^{k}(\textbf{s},2\mu_{1}-1)$  can be reformulated as
follows: Given a sequence
$\textbf{s}=\{s_{i}\}_{i=0}^{2\mu_{1}-1}\in{\cH}$, where
$\mu_1$ is the smallest index $n$ such that~\eqref{eq:odd} holds,
find all functions $f\in N_{\kappa}^{k}$, such that
\begin{equation}\label{3p.3.3a}
    f(z)=-\frac{s_{\nu_{1}-1}}{z^{\nu_{1}}}-\cdots
    -\frac{s_{2\mu_{1}-1}}{z^{2\mu_{1}}}+o\left(\frac{1}{z^{2\mu_{1}}}\right), \quad {z} \wh \to \infty.
\end{equation}

Solution of the basic {even} moment problem will be splitted into
two steps. On the first step one applies Lemma~\ref{lem:D02} to
construct a sequence $\{{\mathfrak
s}^{(1)}_j\}_{j=1}^{2(\mu_j-\nu_j)-1}$. If  $f\in
\mathcal{M}_{\kappa}^{k}(\textbf{s},
  2\mu_1-1)$ then by Theorem~\ref{3p.th3.3}
  $f(z)$ admits the representation~\eqref{3p.3.8'}
which can be rewritten as
   \begin{equation}\label{eq:g_1_TM}
    -\frac{1}{f(z)}=zm_1(z)-\frac{1}{g_1(z)},
\end{equation}
and   where $-g_1^{-1}$ has the following asymptotic expansion
\begin{equation}\label{eq:g_1}
-\frac{1}{g_1(z)}=-{\mathfrak s}^{(1)}_{-1}-\frac{{\mathfrak
s}^{(1)}_{0}}{z}-\cdots-\frac{{\mathfrak
s}^{(1)}_{2(\mu_1-\nu_1)-1}}{z^{2(\mu_1-\nu_1)}}
+o\left(\frac{1}{z^{2(\mu_1-\nu_1)}}\right),\quad
z\widehat{\rightarrow}\infty
\end{equation}
with ${\mathfrak s}^{(1)}_{i}$  defined by~\eqref{eq:2.4G^1}.
Moreover, $f\in{\mathbf N}_\kappa^k$ if and only if $g_1\in{\mathbf
N}_{\kappa-\kappa_-(zm_1)}^{k-\kappa_-(m_1)}$. Now two cases may
occur.
\begin{enumerate}
  \item [(1)] If $\nu_1=\mu_1$, then  ${\mathfrak s}_{-1}^{(1)}\ne 0$ and by Lemma~\ref{lem:D03} $g_1$ admits the representation
\begin{equation}\label{eq:g_1_TL}
    g_1=T_{L_1}[f_1]:=l_1+f_1
\end{equation}
where $l_1$ is a constant
\begin{equation}\label{eq:l_10}
    l_1
    =\frac{1}{{\mathfrak s}_{-1}^{(1)}}=(-1)^{\nu_{1}
    +1}s_{\nu_1-1}\frac{D_{\nu_{1}}}{D_{\nu_{1}}^+},
\end{equation}
$L_1$ is defined by~\eqref{eq:W_l} and $f_1(z)=o(1)$ as $z\widehat{\rightarrow}\infty$. Moreover, by Lemma~\ref{lem:W_l} $g_1\in {\mathbf N}_{\kappa'}^{k'}$ if and only if
$f_1\in {\mathbf N}_{\kappa'}^{k'-\kappa_-(zl_1)}$.
  \item [(2)] If $\nu_1<\mu_1$, then  ${\mathfrak s}^{(1)}= 0$ and by Lemma~\ref{lem:D01} $g_1$ admits the representation~\eqref{eq:g_1_TL},
where $l_1=l_1(z)$ is a polynomial
\begin{equation}\label{eq:l_1}
    l_1(z)
    =\frac{1}{{\mathfrak s}^{(1)}_{\mu_1-\nu_1-1}\det({\mathcal S}^{(1)}_{\mu_1-\nu_1})}
    \begin{vmatrix}
    {\mathfrak s}^{(1)}_{0} & \ldots & {\mathfrak s}^{(1)}_{\mu_{1}-\nu_{1}-1} & {\mathfrak s}^{(1)}_{\mu_{1}-\nu_{1}} \\
    \cdots & \cdots & \cdots & \cdots \\
    {\mathfrak s}^{(1)}_{\mu_{1}-\nu_{1}-1} & \ldots & {\mathfrak s}^{(1)}_{2\mu_{1}-2\nu_{1}-2} & {\mathfrak s}^{(1)}_{2\mu_{1}-2\nu_{1}-1} \\
    1 & \ldots & z^{\mu_{1}-\nu_{1}-1} & z^{\mu_{1}-\nu_{1}} \\
    \end{vmatrix},
\end{equation}
$L_1$ is defined by~\eqref{eq:W_l} and $f_1(z)=o(1)$ as $z\widehat{\rightarrow}\infty$. Moreover, by Lemma~\ref{lem:W_l} $g_1\in {\mathbf N}_{\kappa'}^{k'}$ if and only if
$f_1\in {\mathbf N}_{\kappa'-\kappa_-(l_1)}^{k'-\kappa_-(zl_1)}$.
\end{enumerate}

Combining the formulas~\eqref{eq:g_1_TM} and~\eqref{eq:g_1_TL} and
summarising the above reasonings one obtains the first two
statements of the following

\begin{theorem}\label{3p.th3.5}
Let $\textbf{s}=\{s_{j}\}_{j=0}^{2\mu_{1}-1}$ be a sequence from
${\cH}$, such that $\cN({\mathbf s})=\{\nu_1,\mu_1\}$
$(\nu_1\le\mu_1)$, and let $m_{1}$, $l_{1}$ be defined
by~\eqref{2p.3.10} and \eqref{eq:l_1}, respectively. Then:
\begin{enumerate}
  \item [(1)] The problem $MP_\kappa^k(\textbf{s}, 2\mu_1-1)$ is solvable if and only if
\begin{equation}\label{eq:BasOdd_Solv}
    \kappa_1:=\nu_-(S_{\mu_1})\le\kappa\quad\mbox{and}\quad k_1^+:=\nu_-(S_{\mu_1}^+)\le k.
\end{equation}
  \item [(2)]$f\in \mathcal{M}_{\kappa}^{k}(\textbf{s},
  2\mu_{1}-1)$ if and only if $f$  admits the
  representation~
  \begin{equation}\label{eq:TML_repr}
    f=T_{M_1L_1}[f_1],
  \end{equation}
   where
\begin{equation}\label{eq:tau_1a}
f_1\in
N_{\kappa-\kappa_{1}}^{k-k_{1}^+}\quad\mbox{and}\quad
f_1(z)=o(1)\quad\mbox{as}\quad
z\widehat{\rightarrow}\infty.
\end{equation}
The indices $\kappa_1$ and $k_1^+$ can be expressed in terms of $m_1$ and $l_1$ by \begin{equation}\label{eq:kappa_1k_1_Odd}
    \kappa_{1}=\kappa_{-}(zm_{1})+\kappa_-(l),\quad
     k_{1}^+=\kappa_{-}(m_{1})+\kappa_{-}(zl_{1}).
 \end{equation}
\item [(3)] If $\ell>2\mu_1-1$, then  $f\in \mathcal{M}_{\kappa}^{k}(\textbf{s},
  \ell)$,  if and only if $f$ admits the
  representation~\eqref{eq:TML_repr}, where
  \begin{equation}\label{3p.3.19'a} f_1\in\mathcal{M}_{\kappa-\kappa_{1}}^{k-k_{1}^+}(\textbf{s}^{(1)},
  \ell-2\mu_{1}),
  \end{equation}
$\kappa_1$ and $k_1^+$ are determined by~\eqref{eq:BasOdd_Solv}  and the sequence
$\left\{{s}_{i}^{(1)}\right\}_{i=-1}^{\ell-2\mu_{1}}$
is
determined by the matrix equation
\begin{equation}\label{eq:s^(1)reg}
     T(l_1,-{ s}^{(1)}_{0},\ldots,-
 {s}^{(1)}_{\ell-2\mu_1}) \, T({\mathfrak s}^{(1)}_{-1},\dots,{\mathfrak s}^{(1)}_{\ell-2\mu_1}) = I_{\ell-2\mu_1+2},
\end{equation}
if $\mu_1=\nu_1$, and if $\nu_1<\mu_1$ by the following equation
\begin{equation}\label{eq:s^(1)}
     T(l_{\mu_1-\nu_1}^{(1)},\ldots,l_0^{(1)},-{ s}^{(1)}_{0},\ldots,-
 {s}^{(1)}_{\ell-2\mu_1}) \, T({\mathfrak s}^{(1)}_{\mu_1-\nu_1-1},\dots,{\mathfrak s}^{(1)}_{\ell-2\nu_1}) = I_{\ell-\mu_1-\nu_1+2}.
\end{equation}
\end{enumerate}
\end{theorem}
\begin{proof}
The items (1) and (2) are proved above.

Let us prove (3). Assume that  $\ell>2\mu_1-1$ and  $f\in \mathcal{M}_{\kappa}^{k}(\textbf{s},  \ell)$. Then by Theorem~\ref{3p.th3.3} $f(z)$ admits the representation \eqref{eq:g_1_TM}
   where $-g_1^{-1}$ has the asymptotic expansion
\begin{equation}\label{eq:g_1_ell}
-g_1^{-1}=-{\mathfrak s}^{(1)}_{-1}-\frac{{\mathfrak
s}^{(1)}_{0}}{z}-\cdots-\frac{{\mathfrak
s}^{(1)}_{\ell-2\nu_1}}{z^{\ell-2\nu_1}}
+o\left(\frac{1}{z^{\ell-2\nu_1+1}}\right),\quad
z\widehat{\rightarrow}\infty,
\end{equation}
and ${\mathfrak s}^{(1)}_{i}$ are defined by~\eqref{eq:2.4G^1}.
Moreover, $f\in{\mathbf N}_\kappa^k$ if and only if $g_1\in{\mathbf N}_{\kappa-\kappa_-(zm_1)}^{k-\kappa_-(m_1)}$.
Consider two cases:
\begin{enumerate}
  \item [(1)] If $\nu_1=\mu_1$, then  ${\mathfrak s}^{(1)}\ne 0$ and by Lemma~\ref{lem:D03}
  $g_1$ admits the representation~\eqref{eq:g_1_TL}, \eqref{eq:l_10}, where
$L_1$ is defined by~\eqref{eq:W_l} and $f_1(z)$ has the following
asymptotic
\begin{equation}\label{eq:f_1assymp}
    f_1(z)=-\frac{s_{0}^{(1)}}{z}-\cdots
    -\frac{s^{(1)}_{\ell-2\mu_{1}}}{z^{\ell-2\mu_{1}+1}}+o\left(\frac{1}{z^{\ell-2\mu_{1}+1}}\right),\quad z\widehat{\rightarrow}\infty
\end{equation}
with $s^{(1)}_j$ defined by the matrix equation~\eqref{eq:s^(1)reg}. By Lemma~\ref{lem:W_l}
\[
g_1\in{\mathbf N}_{\kappa-\kappa_-(zm_1)}^{k-\kappa_-(m_1)}
\Longleftrightarrow
f_1\in {\mathbf N}_{\kappa-\kappa_-(zm_1)}^{k-\kappa_-(m_1)-\kappa_-(zl_1)}.
\]
 This proves that $f_1\in\mathcal{M}_{\kappa-\kappa_{1}}^{k-k_{1}^+}(\textbf{s}^{(1)},
  \ell-2\mu_{1})$, since $\kappa_-(l_1)=0$ and $\kappa_1=\kappa_-(zm_1)$ in this case.

  \item [(2)] If $\nu_1<\mu_1$, then  ${\mathfrak s}^{(1)}= 0$ and by Lemma~\ref{lem:D01} $g_1$ admits the representation~\eqref{eq:g_1_TL},
where $l_1=l_1(z)$ is a polynomial given by~\eqref{eq:l_1},
$L_1$ is defined by~\eqref{eq:W_l} and $f_1(z)$ has the asymptotic
\eqref{eq:f_1assymp} as $z\widehat{\rightarrow}\infty$.
By Lemma~\ref{lem:W_l}
\[
g_1\in{\mathbf N}_{\kappa-\kappa_-(zm_1)}^{k-\kappa_-(m_1)}
\Longleftrightarrow f_1\in {\mathbf
N}_{\kappa-\kappa_-(zm_1)-\kappa_-(l_1)}^{k-\kappa_-(m_1)-\kappa_-(zl_1)}.
\]
This proves that $f_1\in\mathcal{M}_{\kappa-\kappa_{1}}^{k-k_{1}^+}(\textbf{s}^{(1)},
  \ell-2\mu_{1})$ also in the case $\nu_1<\mu_1$.
\end{enumerate}
The proof of the converse statement is similar and is based on
Lemmas~\ref{lem:D02}, \ref{lem:D03}, \ref{lem:W_m}, \ref{lem:W_l}.
\end{proof}
\begin{remark}\label{rem:3.6}
It follows from the equality~\eqref{eq:2.4G} and
\cite[Proposition~2.1]{Der03} that the sequence
$\left\{{s}_{i}^{(1)}\right\}_{i=0}^{\ell-2\mu_{1}}$
can be found by the equalities
\begin{equation}\label{eq:s^{2}}
    s^{(1)}_{i}=\frac{(-1)^{i+\mu_{1}-\nu_{1}}}
    {({\mathfrak s}^{(1)}_{\mu_{1}-\nu_{1}-1})^{i+\mu_{1}-\nu_{1}+2}}
        \begin{vmatrix}
            {\mathfrak s}_{\mu_{1}-\nu_{1}}^{(1)} & {\mathfrak s}_{\mu_{1}-\nu_{1}-1}^{(1)} & 0 & \ldots & 0 \\
            \vdots & \ddots & \ddots & \ddots & \vdots \\
            \vdots &   & \ddots & \ddots & 0 \\
            \vdots &   &   & \ddots & {\mathfrak s}_{\mu_{1}-\nu_{1}-1}^{(1)} \\
            {\mathfrak s}_{\mu_{1}-\nu_{1}+i}^{(1)} &  \ldots &  \ldots    &  \ldots & {\mathfrak s}_{\mu_{1}-\nu_{1}}^{(1)} \\
        \end{vmatrix},
 \end{equation}
 where $i=\overline{0,\ell-2\mu_{1}}$.
\end{remark}
\begin{remark}\label{ref:SolMatr}
The solution matrix of the basic {even} moment problem $\mathcal{M}_{\kappa}^{k}(\textbf{s},
  2\mu_{1}-1)$
\begin{equation}\label{3p.3,8a}
W_{2}(z)=\begin{pmatrix}
                                       1 & l_{1}(z) \\
                                       -zm_{1}(z) & -zm_{1}(z)l_{1}(z)+1 \\
                                     \end{pmatrix}
\end{equation}
admits the following factorization
\begin{equation}\label{3p.3,8c}W_{2}(z)=M_{1}(z)L_{1}(z),\end{equation}
where the matrices $M_{1}(z)$ and $L_{1}(z)$ are defined
by~\eqref{eq:W_m}, \eqref{eq:W_l} and    the corresponding linear
fractional transform is defined by
\begin{equation}\label{3p.3,9a}
    T_{W_{2}}[f_{1}]=\frac{f_{1}(z)+l_{1}(z)}{-zm_{1}(z)f_{1}(z)-zm_{1}(z)l_{1}+1}.
\end{equation}
\end{remark}

\section{Schur algorithm.}

\subsection{Regular sequences}
A general nondegenerate indefinite truncated moment problem in the class  $N_{\kappa}^{k}$
can be studied by the step-by-step algorithm based on the elementary steps, introduced in the previous section.
In this section we will demonstrate this algorythm in the case when the sequence $\textbf{s}$ belongs
to the class ${\cH}_{\kappa,\ell}^{k, reg}$ of so-called regular sequences.
This class ${\cH}_{\kappa,\ell}^{k, reg}$ was introduces in \cite{DK15}.

\begin{definition} \label{def:3p.4.3}(\cite{DK15})
Let ${\mathbf s}=\{s_i\}_{i=0}^\ell\in{\cH}_{\kappa,\ell}$ and let $\cN({\mathbf s})=\{n_j\}_{j=1}^N$.
A sequence ${\mathbf s}$ is related to the class ${\cH}_{\kappa,\ell}^{ reg}$ and is said to be regular,
if one of the following equivalent conditions  holds:
\begin{enumerate}
  \item [(1)] $P_{{n}_{j}}(0)\neq0$ for every $j\le N$;
  \item [(2)] $D_{{n}_{j}-1}^{+}\neq0$ for every $j\le N$;
  \item [(3)] $D_{{n}_{j}}^{+}\neq0$ for every $j\le N$;
  \item [(4)] $\nu_{j }=\mu_{j}$   for all $j$,  such that $\nu_j,\mu_j\in\cN({\mathbf s})$.
\end{enumerate}
The equivalence of the conditions $(1)-(4)$ was proved in~\cite[Lemma~3.1]{DK15}.
The class of regular ${\cH}_{\kappa,\ell}^{k}$-sequences is defined by
${\cH}_{\kappa,\ell}^{k,reg}:={\cH}_{\kappa,\ell}^{reg}\cap{\cH}_{\kappa,\ell}^{k}$.
\end{definition}
For a regular sequence
${\mathbf s}\in{\cH}_{\kappa,\ell}^{k,reg}$
the normal indices
$n_j$ $(1\le j\le N)$ of ${\mathbf s}$
satisfy
\[
n_j=\nu_{j }=\mu_{j}\quad (1\le j\le N),
\]
where $\nu_{j }$ and $\mu_{j}$ are introduced in~\eqref{eq:3p.nu}
and~\eqref{eq:3p.mu}. As was shown in \cite{DK15} for every sequence
${\mathbf s}\in{\cH}_{\kappa,\ell}^{k,reg}$ there are
polynomials $m_j$ of degree $\nu_j-n_{j-1}-1$ and real numbers $l_j$
such that the $2j-$th convergent $\frac{u_{2j}}{v_{2j}}$ of the
generalized $S-$fraction
\begin{equation}\label{s4.2}
    \frac{1}{\displaystyle -z m_{1}(z)+\frac{1}{\displaystyle
    l_{1}+\dots\frac{1}{\displaystyle -z m_{j}(z) +\frac{1}{\displaystyle
    l_{j}+\dots}}}}.
\end{equation}
has the following asymptotic expansion
\begin{equation}\label{1ps.2.11}
    f(z)\sim-\frac{s_{0}}{z}-\frac{s_{1}}{z^{2}}-\cdots
    -\frac{s_{2{n}_{j}-1}}{z^{2{n}_{j}}} +
    O\left(\frac{1}{z^{2{n}_{j}+1}}\right),\quad z\widehat{\rightarrow}\infty.
\end{equation}
We will show that the Schur process leads to the same continued fraction and gives descriptions of solutions of {odd} and {even} problems $MP_{\kappa}^{k}(\textbf{s},2n_j-2)$ and $MP_{\kappa}^{k}(\textbf{s},2n_j-1)$ in terms of these continued fractions.

\subsection{{Odd} moment problem} Let
$MP_{\kappa}^{k}(\textbf{s},2n_N-2)$ be a nondegenerate {odd} moment problem, i.e.
\begin{equation}\label{eq:even_S}
    D_{n_N}\ne 0 \quad\mbox{ and }\quad D_{n_N-1}^+\ne 0.
\end{equation}
Assume that $f\in\cM_{\kappa}^{k}(\textbf{s},2n_N-2)$ ($N>1$), i.e. $f\in{\mathbf N}_\kappa^k$ and
\[
    f(z)=-\frac{s_{0}}{z}-\frac{s_{1}}{z^{2}}-\cdots-
    \frac{s_{2n_N-2}}{z^{2n_N-1}}+o\left(\frac{1}{z^{2n_N-1}}\right),\quad z\widehat{\rightarrow}\infty.
\]
Then by Theorem~\ref{3p.th3.5},  the function $f$ can be represented
as
\[
    f(z)=\frac{1}{\displaystyle-z m_{1}(z)+\frac{1}{l_{1}+f_{1}(z)}},
\]
where the polynomial $m_1$ and number $l_1$ are defined by~\eqref{2p.3.10} and
\eqref{eq:l_10}, respectively. Here the function $f_1$ has the asymptotic expansion~\eqref{eq:f_1assymp}
with the sequence ${\mathbf s}^{(1)}=\{s_i^{(1)}\}_{i=1}^{2(n_N-n_1)-2}$ determined consequently by \eqref{eq:2.4G^1}
and \eqref{eq:s^(1)}.
The set of normal indices of the sequence ${\mathbf s}^{(1)}$ is $\cN({\mathbf s}^{(1)})=\{n_j-n_1\}_{j=2}^N$.
Continuing this process and applying Theorem~\ref{3p.th3.5} $N-1$ times
one obtains on each step some function $f_j\in{\mathbf N}_{\kappa-\kappa_j}^{k-k_j}$ $(j=1,\dots,N-1)$
with an induced asymptotic expansion
\[
    f_j(z)=-\frac{s_{0}^{(j)}}{z}-\frac{s_{1}^{(j)}}{z^{2}}-\cdots-
    \frac{s_{2(n_N-n_j)-2}^{(j)}}{z^{2(n_N-n_j)-1}}+
    o\left(\frac{1}{z^{2(n_N-n_j)-1}}\right),\quad z\widehat{\rightarrow}\infty,
\]
such that $f_{j-1}$ has  the following  representation in terms of $f_j$:
\begin{equation}\label{eq:f_i-1}
    f_{j-1}(z)=\frac{1}{\displaystyle-z m_{j}(z)+\frac{1}{l_{j}+f_{j}(z)}}\quad(i=1,\dots,j),
\end{equation}
Here the sequence ${\mathbf
s}^{(j)}=\{s_i^{(j)}\}_{i=1}^{2(n_N-n_j)-2}$ is determined
recursively by \eqref{eq:2.4G^1} and \eqref{eq:s^(1)} and  $m_j$ and
$l_j$ are defined by the formulas
\begin{equation}\label{eq:ml_j}
        m_j(z)=\frac{(-1)^{\nu+1}}{D^{(j-1)}_{\nu}}
        \begin{vmatrix}
            0 & \ldots & 0 & s^{(j-1)}_{\nu-1} & s^{(j-1)}_{\nu} \\
            \vdots &  & \ldots & \ldots & \vdots \\
            s^{(j-1)}_{\nu-1} & \ldots & \ldots & \ldots & s^{(j-1)}_{2\nu_{}-2} \\
            1 & z & \ldots& z^{\nu-2}& z^{\nu-1} \\
        \end{vmatrix},
\end{equation}
where $D^{(j)}_\nu:=\det S^{(j)}_\nu$, $\nu=n_{j}-n_{j-1}$ and
\begin{equation}\label{eq:l_j0}
l_j
   =(-1)^{\nu+1}\frac{D^{(j-1)}_{\nu}}{\left(D^{(j-1)}_{\nu}\right)^+}\quad (j=1,\dots,N-1).
\end{equation}
Let the matrix functions $M_j(z)$ and $L_j(z)$ be defined by
\[
    M_{j}(z)=\begin{pmatrix}
                                       1 & 0 \\
                                       -zm_{j}(z) & 1\\
                                     \end{pmatrix}\quad\mbox{and}\quad
    L_{j}(z)=\begin{pmatrix}
                                       1 & l_{j} \\
                                       0 & 1 \\
                                     \end{pmatrix} \quad (j=1,\dots,N-1).
\]
Then it follows from~\eqref{eq:f_i-1} that
\begin{equation}\label{eq:f_j}
    f_{j-1}(z)=T_{M_j(z)L_j(z)}[f_j(z)]\quad(j=1,\dots,N-1).
\end{equation}

 On the last step we get
the function $f_{N-1}(z)$, which is a solution of the basic moment
problem $MP_\kappa^k({\mathbf s}^{(N-1)},2(n_{N}-n_{N-1})-2)$. By
Theorem~\ref{3p.th3.3},  the function $f_{N-1}(z)$ can be
represented as
\[
    f_{N-1}(z)=\frac{1}{\displaystyle-z m_{N}(z)+\frac{1}{f_{N}(z)}}=T_{M_N(z)}[f_N(z)],
\]
where the polynomial $m_N(z)$ is defined by~\eqref{eq:ml_j} and $f_N(z)$
is a function from ${\mathbf N}_{\kappa-\kappa_N}^{k-k_N}$, such that $f_N(z)^{(-1)}=o(z)$ as $z\wh\to\infty$.

Combining the statements~\eqref{eq:f_i-1} and~\eqref{eq:f_j} and replacing $f_N(z)$ by $\tau(z)$,
one obtains the following
\begin{theorem}\label{2p.pr.alg1}
Let ${\mathbf
s}=\{s_i\}_{i=0}^{2n_N-2}\in\cH_{\kappa,2n_N-2}^{k,reg}$, let
$\cN({\mathbf s})=\{n_j\}_{j=1}^N$, and let $m_j(z)$ and $l_j(z)$
are defined by~\eqref{eq:ml_j} and~\eqref{eq:l_j0}, respectively.
Then:
\begin{enumerate}
  \item [(1)] A nondegenerate {odd} moment problem
$MP_{\kappa}^{k}(\textbf{s},2n_N-2)$
is solvable, if and only if
\begin{equation}\label{eq:Gen_Solv}
    \kappa_N:=\nu_-(S_{n_N})\le\kappa\quad\mbox{and}\quad
    k_N:=\nu_-(S_{n_N-1}^+)\le k.
\end{equation}
\item [(2)] $f\in \mathcal{M}_{\kappa}^{k}(\textbf{s},
  2n_{N}-2)$ if and only if $f$ admits the
  representation
\begin{equation}\label{eq:LFT_W1j}
    f=T_{W_{2N-1}}[\tau],
\end{equation}
   where
\begin{equation}\label{eq:W2N-1}
    W_{2N-1}(z):=M_1(z)L_1\dots (z)L_{N-1}M_N(z)
\end{equation}
and $\tau(z)$ satisfies the conditions
\begin{equation}\label{eq:tau_j}
    \tau \in
    {\mathbf N}_{\kappa-\kappa_{N}}^{k-k_{N}}\quad\mbox{and}\quad
    \frac{\displaystyle1}{\displaystyle\tau(z)}=o(z),\quad
    z\widehat{\rightarrow}\infty.
\end{equation}
    \item [(3)]
The  representation
\eqref{eq:LFT_W1j} can be rewritten as a continued fraction expansion
\[
    f(z)= \frac{1}{\displaystyle -z m_{1}(z)+\frac{1}{\displaystyle
    l_{1}+\frac{1}{\displaystyle -zm_{2}(z)+\cdots+\frac{1}{-zm_{N}(z)+\tau(z)}
    }}}.
\]
\item [(4)] The indices $\kappa_N$ and $k_N$ are related to $m_j$ and $l_j$ by
\[
    \begin{split}
    &\kappa_{N}=\sum\limits_{j=1}^{N}\kappa_{-}(zm_{j})
    ,\quad
    k_{N}=\sum\limits_{j=1}^{N}\kappa_{-}(m_{j})+
    \sum\limits_{j=1}^{N-1}\kappa_{-}(zl_{j}).
    \end{split}
\]
\end{enumerate}
\end{theorem}

\subsection{{Even} moment problem} Let ${\mathbf s}=\{s_i\}_{i=0}^{2n_N-1}\in\cH_{\kappa,2n_N-1}^{k}$,
let $\cN({\mathbf s})=\{n_N\}_{j=1}^N$ and let
$MP_{\kappa}^{k}(\textbf{s},2n_N-1)$ be a nondegenerate {even}
moment problem, i.e.
\begin{equation}\label{eq:even_}
    D_{n_N}\ne 0 \quad\mbox{ and }\quad D_{n_N}^+\ne 0.
\end{equation}
Applying Theorem~\ref{3p.th3.5} $N-1$ times in the same way as in
the {odd} case one obtains the equalities~\eqref{eq:f_i-1} and a
sequence of functions $f_j\in \cM_{\kappa-\kappa_j}^{k-k_j}({\mathbf
s}^{(j)},2(n_{N}-n_{j})-1)$. On the last step we obtain the function
$f_{N-1}(z)$, which is a solution of the basic {even} moment problem
$MP_{\kappa-\kappa_{N-1}}^{k-k_{N-1}}({\mathbf
s}^{(N-1)},2(n_{N}-n_{N-1})-1)$. By Theorem~\ref{3p.th3.5},  the
function $f_{N-1}$ can be  represented as follows:
\begin{equation}\label{eq:f_jOdd}
    f_{N-1}(z)=\frac{1}{\displaystyle-z m_{N}(z)+\frac{1}{l_N+f_{N}(z)}},
\end{equation}
where  $m_N(z)$ and $l_N$ are defined by~\eqref{eq:ml_j}
and~\eqref{eq:l_j0},  and $f_N(z)$ is a function from ${\mathbf
N}_{\kappa_{N-1}-\kappa_-(zm_N)}^{k_{N-1}-\kappa_-(m_N)-\kappa_-(zl_N)}$,
such that ${f_N(z)}=o(1)$ as $z\wh\to\infty$.

Combining the statements~\eqref{eq:f_i-1} and~\eqref{eq:f_jOdd} one obtains the following
\begin{theorem}\label{thm:4.2}
Let ${\mathbf s}=\{s_i\}_{i=0}^{2n_N-1}\in\cH_{\kappa,2n_N-1}^{k,reg}$ and let $\cN({\mathbf s})=\{n_j\}_{j=1}^N$.
\begin{enumerate}
  \item [(1)] A nondegenerate {odd} moment problem
$MP_{\kappa}^{k}(\textbf{s},2n_N-1)$
is solvable, if and only if
\begin{equation}\label{eq:Gen_Solv+}
    \kappa_N:=n_-(S_{n_N})\le\kappa\quad\mbox{and}\quad
    k_N^+:=n_-(S_{n_N}^+)\le k.
\end{equation}
\item [(2)] $f\in \mathcal{M}_{\kappa}^{k}(\textbf{s},
  2n_{N}-1)$ if and only if $f$ admits the
  representation
\begin{equation}\label{eq:LFT_W1j+}
    f=T_{W_{2N}^+}[\tau],
\end{equation}
   where
\begin{equation}\label{eq:W1j}
    W_{2N}(z):=W_{2N-1}(z)L_N=M_1(z)L_1\dots M_N(z)L_N
\end{equation}
and $\tau(z)$ satisfies the conditions
\begin{equation}\label{eq:tau_N+}
    \tau \in
    {\mathbf N}_{\kappa-\kappa_{N}}^{k-k_{N}^+}\quad\mbox{and}\quad
    \frac{\displaystyle1}{\displaystyle\tau(z)}=o(1),\quad
    z\widehat{\rightarrow}\infty.
\end{equation}
    \item [(3)]
The  representation
\eqref{eq:LFT_W1j} can be rewritten as a continued fraction expansion
\[
    f(z)= \frac{1}{\displaystyle -z m_{1}(z)+\frac{1}{\displaystyle
    l_{1}+\cdots+\frac{1}{-zm_{N}(z)+\frac{1}{\displaystyle
    l_{N}+\tau(z)}    }}},
\]
where $m_j(z)$ and $l_j$ are defined by~\eqref{2p.3.10} and~\eqref{eq:l_10}, respectively.
\item [(4)] The indices $\kappa_N$ and $k_N^+$ can be found by
\[
    \begin{split}
    &\kappa_{N}=\sum\limits_{j=1}^{N}\kappa_{-}(zm_{j})
    ,\quad
    k_{N}^+=\sum\limits_{j=1}^{N}k_{-}(m_{j})+
    \sum\limits_{j=1}^{N}\kappa_{-}(zl_{j}).
    \end{split}
\]
\end{enumerate}
\end{theorem}

\section{Solution matrices }
In the case of a regular sequence ${\mathbf s}$ the solution
matrices $W_{2N-1}(z)$ and $W_{2N}(z)$ defined by~\eqref{eq:W2N-1}
and~\eqref{eq:W1j} can be represented explicitly in terms of
polynomials of the first and the second kind.
\subsection{Polynomials of the first and the second kind}
Let $\textbf{s}=\{s_i\}_{i=0}^\ell\in{\cH}_{\kappa,\ell}$ and let the sequence
$\cN({\mathbf s})=\{n_j\}_{j=1}^N$ be extended by $n_{-1}:=-1$, $n_{0}:=0$.
Recall (see ~\cite{Akh}, ~\cite{DD04}) that polynomials of the first and the second kind are defined by
\begin{equation}\label{P3.mom.1}
\begin{split}
    P_{n_{j}}(\lambda)&=\frac{1}{D_{\mathfrak{n}_{j}}}\textup{det}
    \begin{pmatrix}
        \!\!s_{0}\!\!  & s_{1} \!\!  & \cdots \!\!  & s_{n_{j}} \\
        \!\!\cdots \! \! & \cdots\! \!  & \cdots \!\!  & \cdots \\
        \!\!s_{n_{j}-1} \!\!  & s_{n_{j}} \! \!  & \cdots & s_{2n_{j}-1}  \\
        \!\! 1 \!\!  & \lambda\! \!  & \cdots \! \! & \lambda^{n_{j}}
        \\
    \end{pmatrix},\\
    Q_{n_{j}}(\lambda)
    &=\mathfrak{S}_{t}\left(\frac{P_{n_{j}}(\lambda)-P_{n_{j}}(t)}{\lambda-t}\right)
    \quad(j=1,\dots,N),
\end{split}
\end{equation}
where $\mathfrak{S}_t$ is a functional  defined on $\text{span
}\{1,t,\dots,t^\ell\}$ by
 $\mathfrak{S}_{t}(t^i)=s_i$ $(i=0,1,\dots,\ell)$.
As is known (\cite{DD11} see also~\cite{DD04}, \cite{Peh92}), there are  real numbers ${b}_{0}=s_{n_1-1}$, $b_j$,
and  monic polynomials $a_j$  of degree $n_{j+1}-n_{j}$ $(0\le j\le N-1)$, such that the $j$-th convergent
of the continued fraction~(\ref{eq:Pfrac})
 has the asymptotic expansion~\eqref{1ps.2.11} for $j=1,\dots, N$.
The polynomials $P_{n_{j}}(\lambda)$ and
$Q_{n_{j}}(\lambda)$ are solutions of the following difference
equations
\begin{equation}\label{system1}
    b_{j}y_{n_{j-1}}(\lambda)-a_{j}(\lambda)y_{n_{j}}(\lambda)+y_{n_{j+1}}(\lambda)=0\quad (j=1,\dots, N-1)
\end{equation}
subject to the initial conditions
\begin{equation}\label{system1.1}
    P_{-1}(\lambda)\equiv0,\mbox{ }P_{0}(\lambda)\equiv1,\mbox{ }
    Q_{-1}(\lambda)\equiv -1, \mbox{
    }Q_{0}(\lambda)\equiv0.
\end{equation}
It follows from~\eqref{system1} that $P_{n_{j}}(\lambda)$ and
$Q_{n_{j}}(\lambda)$ are monic polynomials of degree $n_j$ and $n_j-n_1$, respectively.
Moreover, the $j$-th convergent
of the continued fraction~(\ref{eq:Pfrac}) takes the form
\[
f^{[j]}(z)=-\frac{Q_{n_{j}}(z)}{P_{n_{j}}(z)}\quad (1\le j\le N-1).
\]

\subsection{System of difference equations and Stieltjes  polynomials}
Let us consider a system of difference equations associated with the continued fraction (\ref{s4.2})
\begin{equation}\label{s4.3}
    \left\{
    \begin{array}{rcl}
    y_{2j}-y_{2j-2}=l_{j}(z)y_{2j-1},\\
    y_{2j+1}-y_{2j-1}=-zm_{j+1}(z)y_{2j}\\
    \end{array}\right.
\end{equation}
If the $j$--th convergent of this continued fraction is denoted by
$\frac{u_{j}}{v_{j}}$, then $u_{j}$, $v_{j}$ can be found as solutions of the
system (see~\cite[Section~1]{Wall})
subject to the following initial conditions
\begin{equation}\label{s4.4}
    u_{-1}\equiv1, \quad u_{0}\equiv0; \qquad v_{-1}\equiv0, \quad u_{0}\equiv1.
\end{equation}
The first two convergents  of the continued fraction (\ref{s4.2}) take the form
\[
   \label{s4.5}
    \frac{u_{1}}{v_{1}}=\frac{1}{-zm_{1}(z)}=T_{M_1}[\infty], \quad
    \frac{u_{2}}{v_{2}}=\frac{l_{1}(z)}{-zl_{1}(z)m_{1}(z)+1}
    =T_{M_1L_1}[0].
\]
Similarly, the $(2j-1)$-th and $(2j)$-th convergents
\[
        \frac{u_{2j-1}}{v_{2j-1}}=T_{W_{2j-1}}[\infty],\quad
        \frac{u_{2j-2}}{v_{2j}}=T_{W_{2j}}[0].
\]
\begin{theorem} \label{thm:5.1}
Let ${\mathbf s}\in\cH_{\kappa,\ell}^{k, reg}$. Then the $2j-$th
convergent $\frac{u_{2j}}{v_{2j}}$ of the generalized $S-$fraction
(\ref{s4.2}) coincides with the $j-$th convergent of the
$P-$fraction \eqref{eq:Pfrac} corresponding to the sequence
${\mathbf s}$. The parameters ${l}_j$ and $m_j(z)$ ($j\in\dZ_+$) of
the generalized $S-$fraction (\ref{s4.2}) are connected with the
parameters $b_j$ and $a_j(z)$ ($j\in\dN$) of the $P-$fraction
\eqref{eq:Pfrac} by the equalities
\begin{equation}\label{s4.9}
    b_{0}=\frac{1}{d_{1}},\quad a_{0}(z)=\frac{1}{d_{1}}\left(zm_{1}(z)-\frac{1}{l_{1}}\right),
\end{equation}
\begin{equation}\label{s4.10}
     b_{j}=\frac{1}{l_{j}^{2}d_{j}d_{j+1}},\quad
     a_{j}(z)=\frac{1}{d_{j+1}}\left(zm_{j+1}(z)-\left(\frac{1}{l_{j}}
    +\frac{1}{l_{j+1}}\right)\right),
\end{equation}
where $d_{j}$ is the leading coefficient of $m_{j}(z)$ $(j=1,\dots,N-1)$.
\end{theorem}
In particular, it follows from \eqref{s4.10} that
\begin{equation}\label{eq:b0j}
    b_0\dots b_j=\frac{1}{d_{j+1}}\prod_{i=1}^{j}\left(\frac{1}{d_il_i}\right)^2\quad (j=1,\dots,N-1).
\end{equation}
\begin{definition}\label{def:St_Pol}
Let $\textbf{s}\in \mathcal{H}_{\kappa,\ell}^{k,reg}$. Define   polynomials $P^{+}_{j}(z)$, $Q_{j}^{+}(z)$ by
\begin{equation}\label{2p.new8.r7}
    \begin{split}
        &P^{+}_{-1}(z)\equiv0,\quad P^{+}_{0}(z)\equiv1,\qquad
        Q^{+}_{-1}(z)\equiv1,\quad Q^{+}_{0}(z)\equiv0,
        \\
        &P^{+}_{2i-1}(z)=\frac{-1}{b_{0}\ldots b_{i-1}}
        \begin{vmatrix}
            P_{n_{i}}(z) & P_{n_{i-1}}(z) \\
            P_{n_{i}}(0) & P_{n_{i-1}}(0)\\
        \end{vmatrix},
         \quad P^{+}_{2i}(z)=\frac{P_{n_{i}}(z)}{P_{n_{i}}(0)},\\&
        Q_{2i-1}^{+}(z)=\frac{1}{b_{0}\ldots b_{i-1}}
        \begin{vmatrix}
            Q_{n_{i}}(z) & Q_{n_{i-1}}(z) \\
            P_{n_{i}}(0) & P_{n_{i-1}}(0)\\
        \end{vmatrix},\quad
        Q^{+}_{2i}(z)=-\frac{Q_{n_{i}}(z)}{P_{n_{i}}(0)}.
    \end{split}
\end{equation}
The polynomials $P^{+}_{j}(z)$, $Q_{j}^{+}(z)$ will be called
\emph{the Stieltjes  polynomials} corresponding to the sequence ${\mathbf s}$.
\end{definition}

\begin{lemma}\label{2p.lem.4.4}(\cite{DK15}, \cite{K14})
Let $P_{n_{j}}(\lambda)$ be the polynomials of the first kind, then
\begin{equation}\label{2p.new8.r11}
    P_{n_{j}}(0)=(-1)^{j}\prod_{i=1}^{j}\frac{1}{d_{i}l_{i}}\quad (j=1,\dots,N-1),
\end{equation}
\begin{equation}\label{eq_Pnj2}
    P_{n_{j}}(0)^2=d_{j+1}\prod_{i=0}^{j}b_{i}\quad (j=1,\dots,N-1)
\end{equation}
\begin{equation}\label{eq:lj}
   P_{n_{j-1}}(0)P_{n_{j}}(0)= -\frac{1}{l_j}\prod_{i=0}^{j-1}b_{i}\quad (j=1,\dots,N-1).
\end{equation}
 \end{lemma}
 \begin{proof}
 The first statement was proved in \cite{K14} (see also \cite[Corollary~4.1]{DK15}). The second statement follows from~\eqref{2p.new8.r11} and~\eqref{eq:b0j}
  \[
      P_{n_{j}}(0)^2=\prod_{i=1}^{j}\frac{1}{(d_{i}l_{i})^2}=
      d_{j+1}\prod_{i=0}^{j}b_{i}\quad (j=1,\dots,N-1).
  \]
  The third statement  is implied by~\eqref{2p.new8.r11}, \eqref{eq:b0j} and the following calculations
 \[
        P_{n_{j-1}}(0)P_{n_{j}}(0)=-\prod_{i=1}^{j}\frac{1}{d_{i}l_{i}}
        \prod_{i=1}^{j-1}\frac{1}{d_{i}l_{i}}= -\frac{1}{d_{j}l_{j}}\prod_{i=1}^{j-1}\left(\frac{1}{d_{i}l_{i}}\right)^2
        =-\frac{1}{l_{j}}\prod_{i=0}^{j-1}b_{i}.
 \]
 \end{proof}
\begin{proposition}\label{prop:5.4}
Let $\textbf{s}\in \mathcal{H}_{\kappa,\ell}^{k,reg}$ and let $P^{+}_{j}(z)$ and $Q_{j}^{+}(z)$
be the Stieltjes  polynomials defined by
 (\ref{2p.new8.r7}).  Then solutions $\{u_j\}_{j=0}^N$ and $\{v_j\}_{j=0}^N$ of the system~\eqref{s4.3}, \eqref{s4.4} take the form
\begin{equation}\label{eq:4.17u}
        u_{j}= Q_{j}^+(z),\quad
        v_{j}= P_{j}^+(z) 
        \quad(j=-1,0,\dots,N).
\end{equation}
\end{proposition}
\begin{proof}
Since by Definition~\ref{def:St_Pol}
\[
    \begin{split}\label{2p.s7.11}
        &P^{+}_{-1}(z)\equiv0,\quad
        P^{+}_{0}(z)\equiv1,\quad Q^{+}_{-1}(z)\equiv1,\quad
        Q^{+}_{0}(z)\equiv0,
    \end{split}
\]
it is necessary to prove the formulas
\begin{equation}
    \begin{split}\label{eq:P_j+}
        &P^{+}_{2i-1}(z)=-zm_{i}(z)P^{+}_{2i-2}(z)+P^{+}_{2i-3}(z),\\&
        P^{+}_{2i}(z)=l_{i}P^{+}_{2i-1}(z)+P^{+}_{2i-2}(z)\qquad (j=1,\dots,N),
    \end{split}
\end{equation}
\begin{equation}
    \begin{split}\label{eq:Q_j+}
        &Q^{+}_{2i-1}(z)=-zm_{i}(z)Q^{+}_{2i-2}(z)+Q^{+}_{2i-3}(z),\\&
        Q^{+}_{2i}(z)=l_{i}Q^{+}_{2i-1}(z)+Q^{+}_{2i-2}(z)\qquad
        (j=1,\dots,N).
    \end{split}
\end{equation}

First, we prove the formula (\ref{eq:P_j+}). Calculating $P^{+}_{1}(z)$ and $P^{+}_{2}(z)$, and using~\eqref{system1}, ~\eqref{system1.1} and~\eqref{s4.9} we get
\[
    \begin{split}\label{2p.s7.12}
        P^{+}_{1}(z)&=-{b}_{0}^{-1}\begin{vmatrix}
        P_{n_{1}}(z) & P_{n_{0}}(z) \\
        P_{n_{1}}(0) & P_{n_{0}}(0)\\
        \end{vmatrix}=-d_{1}\begin{vmatrix}
        \frac{\displaystyle zm_{1}(z)}{\displaystyle d_{1}} -\frac{\displaystyle1}{\displaystyle d_{1}l_{1}}& 1 \\
        -\frac{\displaystyle 1}{\displaystyle d_{1}l_{1}}& 1\\ \end{vmatrix}
    \\&=-zm_{1}(z)=-zm_{1}(z)P^{+}_{0}(z)+P^{+}_{-1}(z),
    \end{split}
\]
\[
    \begin{split}\label{2p.s7.12'}
        P^{+}_{2}(z)&=\frac{P_{n_{1}}(z)}{P_{n_{1}}(0)}=
        \frac{\frac{\displaystyle zm_{1}(z)}{\displaystyle d_{1}}
        -\frac{\displaystyle1}{\displaystyle
        d_{1}l_{1}}}{-\frac{\displaystyle1}{\displaystyle
        d_{1}l_{1}}}=-l_{1}zm_{1}(z)+1=l_{1}P^{+}_{1}(z)+P^{+}_{0}(z).
    \end{split}
\]
Next, by~\eqref{system1}, ~\eqref{system1.1} and~\eqref{s4.9} one gets for $i=\overline{1,N}$
\begin{equation*}
    \begin{split}\label{2p.s7.14}
        P^{+}_{2i-1}(z)
        &=-\frac{1}{b_{0}\ldots b_{i-1}}
        \begin{vmatrix}
            P_{n_{i}}(z) & P_{n_{i-1}}(z) \\
            P_{n_{i}}(0) & P_{n_{i-1}}(0)\\
        \end{vmatrix}\\
        &=-\frac{1}{b_{0}\ldots b_{i-1}}
        \begin{vmatrix}
            (\frac{z m_{i}(z)}{d_i}+a_{i-1}(0))P_{n_{i-1}}(z)-b_{i-1}P_{n_{i-2}}(z)& P_{n_{i-1}}(z) \\
            P_{n_{i}}(0)& P_{n_{i-1}}(0)
        \end{vmatrix}\\
        &=
        -zm_{i}(z)P_{n_{i-1}}(z)\frac{P_{n_{i-1}}(0)}{d_ib_0\dots b_{i-1}}\\
        &-
        \frac{P_{n_{i-1}}(z)( a_{i-1}(0)P_{n_{i-2}}(0)-P_{n_{i}}(0))-b_{i-1}P_{n_{i-2}}(z)
        P_{n_{i-1}}(0)}{b_{0}\ldots b_{i-1}}
    \end{split}
\end{equation*}
Using~\eqref{eq_Pnj2} and~\eqref{system1} one obtains
\[
        \frac{P_{n_{i-1}}(0)}{d_ib_0\dots b_{i-1}}=\frac{1}{P_{n_{i-1}}(0)}, \quad
        a_{i-1}(0)P_{n_{i-2}}(0)-P_{n_{i}}(0)=b_{i-1}P_{n_{i-2}(0)}
\]
and hence by~\eqref{2p.new8.r7}
\[
          P^{+}_{2i-1}(z)=-zm_{i}(z)P_{2i-2}^{+}(z)+P_{2i-3}^{+}(z).
\]
This proves the first equality in (\ref{eq:P_j+}).
The second equality in (\ref{eq:P_j+}) is immediate from Definition~\ref{def:St_Pol} and \eqref{eq:lj}. Indeed,
\begin{equation*}
    \begin{split}\label{2p.s7.15}
        P^{+}_{2i}(z)-P^{+}_{2i-2}(z)&
        =\frac{P_{n_{i}}(z)}{P_{n_{i}}(0)}-\frac{P_{n_{i-1}}(z)}{P_{n_{i-1}}(0)}=
        \frac{1}{P_{n_{i}}(0)P_{n_{i-1}}(0)}
        \begin{vmatrix}
                P_{n_{i}}(z) & P_{n_{i-1}}(z) \\
                P_{n_{i}}(0) & P_{n_{i-1}}(0)\\
            \end{vmatrix}\\&
        =\frac{-l_i}{b_0\dots b_{i-1}}
        \begin{vmatrix}
                P_{n_{i}}(z) & P_{n_{i-1}}(z) \\
                P_{n_{i}}(0) & P_{n_{i-1}}(0)\\
        \end{vmatrix}=
        l_{i}P^{+}_{2i-1}(z).
    \end{split}
\end{equation*}

Similarly, one proves the recurrence formula~\eqref{eq:Q_j+}.
\end{proof}
\begin{remark}
Notice, that the formula~\eqref{eq:lj}  corresponds to the formula
(4.20) in \cite[Corollary~4.1]{DK15}, which contains a misprint. The
formulas~\eqref{2p.new8.r11} with formally different coefficients
$\gamma_i=(l_iP_{{n}_{i}}(0)P_{{n}_{i-1}}(0))^{-1}$
were found in~\cite[Corollary~4.1]{DK15}.
However, these formulas coincide, since by~\eqref{eq:lj}
\[
l_{i}P_{{n}_{i}}(0)P_{{n}_{i-1}}(0)=-b_{0}...b_{i-1}.
\]
\end{remark}

\subsection{{Odd} case}
Let ${\bf s}=\{s_i\}_{i=0}^{2n_N-2}$ and let $\cN({\bf s})=\{n_j\}_{j=1}^N$ be the set of
normal indices of the sequence ${\bf s}$. Then polynomials $m_j$ $(1\le j\le N)$
and numbers $l_j$ $(1\le j\le N-1)$ are well defined by the formulas~\eqref{eq:ml_j},
and thus the polynomials $P_j^+$ and $Q_j^+$ $(1\le j\le 2N-1)$ can be computed by the
formulas~\eqref{eq:P_j+} and~\eqref{eq:Q_j+}. Notice, that the polynomials  $P_{n_N}$ and $Q_{n_N}$
cannot be calculated by the formulas~\eqref{P3.mom.1} unless the sequence ${\bf s}=\{s_i\}_{i=0}^{2n_N-2}$
is completed by one more number $s_{2n_N-1}$.
As follows from Proposition~\ref{prop:5.4} the choice of $s_{2n_N-1}$ does not impact the coefficients of
$P_{2N-1}^+$ and $Q_{2N-1}^+$ in~\eqref{2p.new8.r7}, but does impact $P_{2N}^+$ and $Q_{2N}^+$.
\begin{lemma}\label{2p.newth1}
Let ${\bf s}=\{s_i\}_{i=0}^{2n_N-2}\in\cH_{\kappa,2n_N-2}^k$,
$\cN({\bf s})=\{n_j\}_{j=1}^N$ $(\kappa,k\in\dZ_+, N\in\dN)$, let
polynomials $m_j(z)$ $(1\le j\le N)$ and numbers $l_j$ $(1\le j\le
N-1)$ be defined by~\eqref{eq:ml_j}, and let the matrices $M_j(z)$,
$L_j$ and $W_{2j-1}(z)$ be defined by
\begin{equation}\label{eq:LMj'}
 M_{j}(z)=\begin{pmatrix}
1 & 0 \\
-zm_{j}(z) & 1 \\
\end{pmatrix},\quad
L_{j}=\begin{pmatrix}
1 & l_{j} \\
0 & 1 \\
\end{pmatrix}\quad j=\overline{1,N}.
\end{equation}
\begin{equation}\label{eq:W2j-1}
W_{2j-1}(z)=M_{1}(z)L_{1}\ldots L_{j-1}M_{j}(z),
\end{equation}
  Then the matrix $W_{2j-1}(z)$  admits the following
representation
\begin{equation}\label{eq:W_2j-1}
W_{2j-1}(z)=
    \begin{pmatrix}
        Q^{+}_{2j-1}(z) & Q^{+}_{2j-2}(z) \\
        P^{+}_{2j-1}(z) & P^{+}_{2j-2}(z) \\
    \end{pmatrix}\quad (j=1,\dots, N).
\end{equation}
where  the polynomials $P_j^+(z)$ and $Q_j^+(z)$ $(0\le j\le 2N-1)$ are defined either by~\eqref{eq:P_j+}
and~\eqref{eq:Q_j+}, or by~\eqref{2p.new8.r7}.
\end{lemma}
\begin{proof}
Let us prove~(\ref{eq:W_2j-1}) by induction. If $j=1$, then
$W_{1}(z)=M_{1}(z)$ and hence
\begin{equation*}\label{s7.18yyy}
W_{1}(z)=\begin{pmatrix}
        1 & 0 \\
        -zm_{j}(z) & 1 \\
    \end{pmatrix}=
    \begin{pmatrix}
        Q^{+}_{1}(z) & Q^{+}_{0}(z) \\
        P^{+}_{1}(z) & P^{+}_{0}(z) \\
    \end{pmatrix}.
\end{equation*}

Assume that~(\ref{eq:W_2j-1}) holds for some $j<N$ and let us prove~\eqref{eq:W_2j-1} for $j:=j+1$.
By~\eqref{eq:P_j+} and~\eqref{eq:Q_j+}
\begin{equation*}
    \begin{split}
        W_{2j+1}(z)&=M_{1}(z)L_{1}\ldots L_{j}M_{j+1}(z)=W_{2j-1}(z) L_{j}M_{j+1}(z)=\\&=
        \begin{pmatrix}
        Q_{2j-1}^{+}(z)& Q_{2j-2}^{+}(z) \\
        P_{2j-1}^{+}(z) &P_{2j-2}^{+}(z) \\
        \end{pmatrix}
        \begin{pmatrix}
            1 & l_{j} \\
            0 & 1 \\
        \end{pmatrix}
        \begin{pmatrix}
            1 & 0 \\
            -zm_{j+1}(z) & 1 \\
        \end{pmatrix}
        =
    \\
    &=
        \begin{pmatrix}
            Q_{2j-1}^{+}(z)& Q_{2j}^{+}(z) \\
            P_{2j-1}^{+}(z) &P_{2j}^{+}(z) \\
        \end{pmatrix}
        \begin{pmatrix}
            1 & 0 \\
            -zm_{j+1}(z) & 1 \\
        \end{pmatrix}=\\
    &=
        \begin{pmatrix}
            Q_{2j+1}^{+}(z)& Q_{2j}^{+}(z) \\
            P_{2j+1}^{+}(z) &P_{2j}^{+}(z) \\
        \end{pmatrix}.
    \end{split}
\end{equation*}
This completes the proof.
\end{proof}

Combining Theorem~\ref{2p.pr.alg1} and Lemma~\ref{2p.newth1} one obtains the following
\begin{theorem}\label{thm:DescrEven}
Let ${\mathbf s}=\{s_i\}_{i=0}^{2n_N-2}\in\cH_{\kappa,2n_N-2}^{k,reg}$, $\cN({\mathbf s})=\{n_j\}_{j=1}^N$ $(\kappa,k\in\dZ_+, N\in\dN)$ and let the polynomials $P_j^+(z)$ and $Q_j^+(z)$ $(0\le j\le 2N-1)$ be defined by~\eqref{eq:P_j+} and~\eqref{eq:Q_j+}. Then:
\begin{enumerate}
  \item [(1)] A nondegenerate {odd} moment problem
$MP_{\kappa}^{k}(\textbf{s},2n_N-2)$
is solvable, if and only if
\begin{equation}\label{eq:Gen_Solv'''}
    \kappa_N:=\nu_-(S_{n_N})\le\kappa\quad\mbox{and}\quad
    k_N:=\nu_-(S_{n_N-1}^+)\le k.
\end{equation}
\item [(2)] $f\in \mathcal{M}_{\kappa}^{k}(\textbf{s},
  2n_{N}-2)$ if and only if $f$ admits the
  representation
\begin{equation}\label{2p.new6}
f(z)=\frac{Q^{+}_{2N-1}(z)\tau(z)+Q^{+}_{2N-2}(z)}{P^{+}_{2N-1}(z)\tau(z)+P^{+}_{2N-2}(z)},
\end{equation}
where
\begin{equation}\label{2p.new17}
    \tau\in N^{k-k_{N}}_{\kappa-\kappa_{N}}\quad\mbox{and}\quad\frac{1}{\tau(z)}=o(z),\quad
    z\widehat{\rightarrow}\infty.
\end{equation}
\end{enumerate}
\end{theorem}
\subsection{Even case} Consider now the case when
${\bf s}=\{s_i\}_{i=0}^{\ell}\in\cH_{\kappa,\ell}^k$ and $\ell$ is odd, i.e. $\ell=2n_N-1$ and $n_N$ is the largest normal index of ${\bf s}$,  $\cN({\bf s})=\{n_j\}_{j=1}^N$. Then polynomials $m_j$ and numbers $l_j$ $(1\le j\le N-1)$ are well defined by the formulas~\eqref{eq:ml_j} for all $ j=\overline{1, N}$, and thus the polynomials $P_j^+$ and $Q_j^+$ $(1\le j\le 2N)$ can be computed by~\eqref{eq:P_j+} and~\eqref{eq:Q_j+}. In the {even} case the polynomials  $P_{n_j}$ and $Q_{n_j}$ are also well defined by the formulas~\eqref{P3.mom.1} for all $ j=\overline{1, 2N}$ and, thus, $P_{j}^+$ and $Q_{j}^+$ for $ j=\overline{1, N}$ can be computed by the formulas~\eqref{2p.new8.r7} as well.
\begin{lemma}\label{lem:W2j}
Let ${\bf s}=\{s_i\}_{i=0}^{2n_N-1}\in\cH_{\kappa,2n_N-1}^k$,
$\cN({\bf s})=\{n_j\}_{j=1}^N$ $(\kappa,k\in\dZ_+, N\in\dN)$, let
$m_j(z)$,  $l_j$, $M_j(z)$ and $L_j$ $(1\le j\le N-1)$ be defined
by~\eqref{eq:ml_j}, \eqref{eq:LMj'} and let the matrix $W_{2j}(z)$
be defined by
\begin{equation}\label{eq:W2j}
W_0(z)=I,\quad W_{2j}(z)=M_{1}(z)L_{1}\ldots M_{j}(z)L_{j} \quad(j=\overline{0,N}).
\end{equation}
  Then the matrix $W_{2j}(z)$  admits the following representation
\begin{equation}\label{eq::W_2j-1}
W_{2j}(z)=
    \begin{pmatrix}
        Q^{+}_{2j-1}(z) & Q^{+}_{2j}(z) \\
        P^{+}_{2j-1}(z) & P^{+}_{2j}(z) \\
    \end{pmatrix}\quad (j=0,\dots, N),
\end{equation}
where  the polynomials $P_j^+(z)$ and $Q_j^+(z)$ $(-1\le j\le 2N)$ are defined  by~\eqref{2p.new8.r7}.
\end{lemma}
\begin{proof}
Let us prove~(\ref{eq:W2j}) by induction.
If $j=0$, then $W_{0}=I$ and hence by~\eqref{2p.new8.r7}
\[
W_{0}=\begin{pmatrix}
        1 & 0 \\
        0 & 1 \\
    \end{pmatrix}=
    \begin{pmatrix}
        Q^{+}_{-1}(z) & Q^{+}_{0}(z) \\
        P^{+}_{-1}(z) & P^{+}_{0}(z) \\
    \end{pmatrix}.
\] 

Assume that~(\ref{eq:W2j}) holds for some $j-1<N$. Then by~\eqref{eq:P_j+} and~\eqref{eq:Q_j+}
\[
\begin{split}
    W_{2j}(z)&=M_{1}(z)L_{1}\ldots M_{j}(z)L_{j}=W_{2j-2}(z) M_{j}(z)L_{j}=\\&=
    \begin{pmatrix}
    Q_{2j-3}^{+}(z)& Q_{2j-2}^{+}(z) \\
    P_{2j-3}^{+}(z) &P_{2j-2}^{+}(z) \\
    \end{pmatrix}
    \begin{pmatrix}
        1 & 0 \\
        -zm_{j}(z) & 1 \\
    \end{pmatrix}
    \begin{pmatrix}
        1 & l_{j} \\
        0 & 1 \\
    \end{pmatrix}
    =\\
&=
    \begin{pmatrix}
        Q_{2j-1}^{+}(z)& Q_{2j-2}^{+}(z) \\
        P_{2j-1}^{+}(z) &P_{2j-2}^{+}(z) \\
    \end{pmatrix}
    \begin{pmatrix}
         1 & l_{j} \\
        0 & 1 \\
    \end{pmatrix}=\\
&=
    \begin{pmatrix}
        Q_{2j-1}^{+}(z)& Q_{2j}^{+}(z) \\
        P_{2j-1}^{+}(z) &P_{2j}^{+}(z) \\
    \end{pmatrix}.
\end{split}
\]
This proves (\ref{eq:W2j}).
\end{proof}

Combining Theorem~\ref{thm:4.2} and Lemma~\ref{lem:W2j} one obtains the following
\begin{theorem}\label{thm:DescrOdd}
Let ${\mathbf s}=\{s_i\}_{i=0}^{2n_N-1}\in\cH_{\kappa,2n_N-1}^{k,reg}$, $\cN({\mathbf s})=\{n_j\}_{j=1}^N$ $(\kappa,k\in\dZ_+, N\in\dN)$ and let the polynomials $P_j^+(z)$ and $Q_j^+(z)$ $(0\le j\le 2N)$ be defined by~\eqref{eq:P_j+} and~\eqref{eq:Q_j+}. Then:
\begin{enumerate}
  \item [(1)]
  The {even}
  moment problem
$MP_{\kappa}^{k}(\textbf{s},2n_N-1)$
is solvable, if and only if
\begin{equation}\label{eq:Gen_Solv+2}
    \kappa_N:=n_-(S_{n_N})\le\kappa\quad\mbox{and}\quad
    k_N^+:=n_-(S_{n_N}^+)\le k.
\end{equation}
\item [(2)] $f\in \mathcal{M}_{\kappa}^{k}(\textbf{s},
  2n_{N}-1)$ if and only if $f$ admits the
  representation
\begin{equation}\label{2p.new6'}
f(z)=\frac{Q^{+}_{2N}(z)+Q^{+}_{2N-1}(z)\tau(z)}{P^{+}_{2N}(z)+P^{+}_{2N-1}(z)\tau(z)},
\end{equation}
where
\begin{equation}\label{2p.new17'}
    \tau\in N_{\kappa-\kappa_{N}}^{k-k_{N}^+}
    \quad\mbox{and}\quad\tau(z)=o(1),\quad
    z\widehat{\rightarrow}\infty.
\end{equation}
\end{enumerate}
\end{theorem}

\end{document}